\documentclass[11pt]{article}

\usepackage[T1]{fontenc}
\usepackage[utf8]{inputenc}
\usepackage[margin=1.05in]{geometry}
\usepackage{amsmath,amssymb,amsthm,mathtools}
\usepackage{microtype}
\usepackage[hidelinks]{hyperref}
\hypersetup{pdftitle={Generalized Einstein Laurent Polynomials, Toric
  Kahler-Einstein Rigidity, and Finite Exponential Families},
  pdfauthor={Shaosai Huang}}

\newtheorem{theorem}{Theorem}[section]
\newtheorem{proposition}[theorem]{Proposition}
\newtheorem{lemma}[theorem]{Lemma}
\newtheorem{corollary}[theorem]{Corollary}
\newtheorem{mainthm}{Theorem}

\newtheorem{maincor}[mainthm]{Corollary}
\theoremstyle{definition}
\newtheorem{definition}[theorem]{Definition}
\newtheorem{remark}[theorem]{Remark}

\newcommand{\C}{\mathbb C}
\newcommand{\R}{\mathbb R}
\newcommand{\Z}{\mathbb Z}
\newcommand{\NP}{\operatorname{Newt}}
\newcommand{\ord}{\operatorname{ord}}
\newcommand{\Ric}{\operatorname{Ric}}
\newcommand{\Aut}{\operatorname{Aut}}
\newcommand{\FS}{\mathrm{FS}}

\newcommand{\Span}{\operatorname{span}}
\newcommand{\conv}{\operatorname{conv}}
\newcommand{\supp}{\operatorname{supp}}
\newcommand{\id}{\mathrm{id}}

\title{Generalized Einstein Laurent Polynomials,\\
Toric K\"ahler--Einstein Rigidity, and Finite Exponential
Families\thanks{Working paper. Comments welcome.}}
\author{Shaosai Huang\thanks{Kspectra Research Inc., Toronto, ON, M2N 0G3,
Canada. Email:
\href{mailto:arthur.foxie.huang@kspectra.ai}{arthur.foxie.huang@kspectra.ai}.}}
\date{\today}

\begin{document}
\maketitle

\begin{abstract}
We study when the logarithm $\psi$ of a positive finite exponential sum on
$\R^d$ satisfies $\det\nabla^2\psi=C\exp(\langle b,\theta\rangle-\lambda\psi)$.
This is the K\"ahler--Einstein equation for metrics induced by exponential
maps into projective space and, for natural exponential families, the
condition that the Jeffreys prior be a Diaconis--Ylvisaker conjugate prior.
First, we classify the bivariate Laurent polynomials with unimodular support
satisfying the generalized Einstein condition of Di Scala and Sombra: up to
units and monomial changes of coordinates, they are powers of an affine
trinomial or products of powers of two independent binomials.  Second, we
show in every dimension that a smooth compact toric manifold with a
projectively induced K\"ahler--Einstein metric is a product of projective
spaces with matched multiples of the Fubini--Study metrics, immersed by a
complete Veronese--Segre system up to automorphisms.  This proves the compact
toric case of the homogeneity conjecture for such metrics and the fixed-point
germ and univalent forms of a conjecture of Manno and Salis.  Third, without
lattice or rationality assumptions, the finite-support exponential families
satisfying the equation are, up to affine changes of statistic, exactly the
products of multinomial families with a common ratio of categories to trials;
this settles the finite-support case of a question of Casalis.
\end{abstract}



\section{Introduction}

This paper studies one rigidity phenomenon from three points of view:
algebraic, geometric, and statistical.  Let $A\subset\R^d$ be a finite set
whose affine span is $\R^d$, let $c_a>0$ for $a\in A$, and put
\[
 Z(\theta)=\sum_{a\in A}c_ae^{\langle a,\theta\rangle},
 \qquad \psi=\log Z.
\]
We ask when there are constants $C>0$, $b\in\R^d$, and $\lambda\in\R$ such
that
\begin{equation}
 \det\nabla^2\psi(\theta)
 =C\exp\{\langle b,\theta\rangle-\lambda\psi(\theta)\}
 \qquad(\theta\in\R^d).
 \label{eq:intro-fisher-determinant}
\end{equation}
If the coordinates of $\theta$ are split into blocks
$(\theta_{j1},\ldots,\theta_{jn_j})$, $1\le j\le k$, then
$Z(\theta)=\prod_j(1+\sum_qe^{\theta_{jq}})^{m_j}$ is a solution whenever every
ratio $(n_j+1)/m_j$ equals $\lambda$.  We show that, up to affine changes of
$\theta$, exponential tilts, and rescaling, there are no other solutions, and
we establish the corresponding rigidity for the projective geometry behind the
equation.

Equation~\eqref{eq:intro-fisher-determinant} has two classical readings.  In
K\"ahler geometry, the holomorphic map
$\Phi(z)=[\sqrt{c_a}\,e^{\langle a,z\rangle}]_{a\in A}$ from $\C^d$ to
$\mathbb P^{|A|-1}$ pulls the Fubini--Study form back to a K\"ahler form
$\omega$ with potential $\psi(z+\bar z)$, and
\eqref{eq:intro-fisher-determinant} is exactly the K\"ahler--Einstein
equation $\Ric(\omega)=\lambda\omega$.  If $A$ lies in a lattice, $\Phi$ is a
monomial map of the complex torus.  Up to an automorphism and a unitary change
of target coordinates, every K\"ahler--Einstein metric induced on a compact
toric manifold by a projective immersion has this form on the dense orbit
(Lemma~\ref{lem:adapted-torus}).  Loi and Zedda conjecture that complete
K\"ahler--Einstein manifolds K\"ahler immersed into a finite-dimensional
projective space are homogeneous \cite[Conjecture~4.3.3]{LoiZedda2018}; for
compact toric manifolds this was known through complex dimension six
\cite{ArezzoLoiZuddas2012,MannoSalis2026} and for selected families
\cite{DiScalaSombra2025}.

In statistics, $\psi$ is the cumulant function of
the natural exponential family generated by $\sum_ac_a\delta_a$, its Hessian
is the Fisher information, and \eqref{eq:intro-fisher-determinant} says that
the Jeffreys prior is a Diaconis--Ylvisaker conjugate prior
\cite{DiaconisYlvisaker1979}.  The same identity characterizes the families
for which the Diaconis--Ylvisaker family, transported to the mean parameter,
coincides with the analogous family defined directly on the mean parameter
\cite{Casalis1996,GutierrezPenaSmith1995}.  In dimension one these are exactly
Morris's quadratic-variance families
\cite{Morris1982,ConsonniVeronese1992,EavesChang1992}; on $\R^d$, Casalis
observed that simple quadratic and Wishart families have this property but
could not determine the whole class \cite{Casalis1996}.

Our main results form a chain.  Theorem~\ref{thm:intro-surface} classifies
the relevant Laurent polynomials in two variables.
Theorem~\ref{thm:intro-global} and Corollary~\ref{cor:intro-manno-salis} use
it to prove, in every dimension, the compact toric case of the Loi--Zedda
conjecture and the germ and univalent forms of conjectures of Manno and Salis
\cite{MannoSalis2024}.
Theorem~\ref{thm:intro-nef} then removes the lattice, the torus, and the
compact manifold from the hypotheses: for an arbitrary finite carrier $A$, the
solutions of \eqref{eq:intro-fisher-determinant} are, up to the symmetries
above, exactly the matched multinomial products.  In statistical language,
these are the only finitely supported natural exponential families whose
Jeffreys prior is a Diaconis--Ylvisaker conjugate prior, or equivalently whose
two conjugate families coincide.  A single two-dimensional obstruction thus
governs a problem in K\"ahler geometry and a question of Casalis in
statistics.

\subsection{Generalized Einstein polynomials in two variables}

We begin with the algebraic core of the paper.  When $A\subset\Z^d$, the
substitution $x=e^\theta$ turns $Z$ into a positive
Laurent polynomial $p$ and $\partial_{\theta_i}$ into
$\Theta_i=x_i\partial_{x_i}$.  For a Laurent polynomial $p$ of full effective
rank, with arbitrary complex coefficients, the logarithmic Monge--Amp\`ere
polynomial is $\mu(p)=p^{d+1}\det\Theta^2\log p$ (see \eqref{eq:def-mu}), and
\eqref{eq:intro-fisher-determinant} becomes $\mu(p)=Cx^bp^{\,d+1-\lambda}$ on
the positive orthant.  Di Scala and Sombra introduced the resulting
\emph{generalized Einstein condition} (GEC) $\mu(p)\mid p^N$, which holds
whenever $p$ comes from a K\"ahler--Einstein toric immersion of a compact
toric manifold (Section~\ref{sec:global}), and proved that it passes to every
face of the Newton polytope \cite{DiScalaSombra2025}.  GEC is formulated for unimodular supports: at each
vertex $v$ of the convex hull, the nearest other support points on the
incident edges differ from $v$ by a lattice basis
\cite[Definition~2.5]{DiScalaSombra2025}; supports of full toric immersions
of compact toric manifolds have this property
\cite[Lemma~2.6]{DiScalaSombra2025}.

By heredity, every two-face carries a
rank-two GEC polynomial.  The rank-two edge identities of Di Scala and Sombra
do not determine these polynomials: they are compatible with smooth polygons
having six or more sides.  Our first main result gives the complete
classification.

\begin{mainthm}[Bivariate classification]
\label{thm:intro-surface}
Let $p\in\C[\Z^2]$ have effective rank two and unimodular support.  If
$\mu(p)\mid p^N$ for some integer $N>0$, then, up to a Laurent unit and an
integral monomial change of coordinates,
\[
 p=(\alpha_0+\alpha_1x+\alpha_2y)^m
 \qquad\text{or}\qquad
 p=(\alpha_0+\alpha_1x)^r(\beta_0+\beta_1y)^s,
\]
with $m,r,s>0$ and all displayed coefficients nonzero.  Conversely, every such
polynomial satisfies $\mu(p)\mid p^3$.
\end{mainthm}

The proof rests on a signed factor-valuation ledger.  Write
$p=\prod_\nu f_\nu^{m_\nu}$ with pairwise nonassociate irreducible $f_\nu$, so
that GEC gives $\mu(p)=\prod_\nu f_\nu^{e_\nu}$, both up to Laurent units, and
put $r_\nu=3m_\nu-e_\nu$.  For every edge $E$ of $P=\NP(p)$, with invariant
curve $C_E$ on the toric surface $X_P$, Proposition~\ref{prop:ledger} gives
\[
 2+C_E^2=\sum_\nu r_\nu\,\ell_E\bigl(\NP(f_\nu)\bigr),
\]
where $\ell_E(Q)$ is the lattice length of the face of $Q$ exposed by the
inner normal of $E$.  The defects $r_\nu$ have no sign in general, but exact
valuations show that two-dimensional irreducible factors have defect three,
while binomial factors meeting the rest of $p$ transversely have defect two
(Section~\ref{sec:valuations}).  If $P$ has at least six sides, then $X_P$
has at least three invariant $(-1)$-curves.  Modulo three, the ledger forces a
binomial factor parallel to each of the corresponding edges; two of these
binomials are nonparallel, hence transverse with defect two, and the ledger at
a $(-1)$-edge becomes $1=3A_E+2$ with an integer $A_E\ge0$, which is
impossible.  Noether's formula and a common-sign relation for the
self-intersections exclude the remaining polygons other than triangles and
parallelograms, and the ledger then recovers the factors
(Section~\ref{sec:polygon}).

Theorem~\ref{thm:intro-surface} also connects with algebraic statistics.  In
the unimodular rank-two setting, GEC implies maximum-likelihood degree one but
not conversely (Proposition~\ref{prop:gec-ml-degree}).  The theorem is
therefore related to, but distinct from, the classification of polygons with
linear precision
\cite{GarciaPuenteSottile2010,BothmerRanestadSottile2010,ClarkeCox2020}.

\subsection{Compact toric K\"ahler--Einstein rigidity}

We now return to the geometric question that motivated the classification.
Throughout, we use the convention
$[\omega_{\FS}]/2\pi=c_1(\mathcal O_{\mathbb P^s}(1))$, so that
$\Ric(\omega_{\FS})=(s+1)\omega_{\FS}$.  In the following theorem the induced
metric is not assumed to be invariant under the given torus, and the
immersion is not assumed to be defined by a complete linear series.

\begin{mainthm}[Compact toric K\"ahler--Einstein rigidity]
\label{thm:intro-global}
Let $X$ be a smooth compact toric manifold of complex dimension $d$, and let
$\varphi\colon X\to\mathbb P^s$ be a full holomorphic immersion such that
$\omega=\varphi^*\omega_{\FS}$ satisfies $\Ric(\omega)=\lambda\omega$ for some
$\lambda\in\R$.  Then $\lambda>0$, and there are positive integers
$n_1,\ldots,n_k$ and $m_1,\ldots,m_k$ with $\sum_jn_j=d$ and
$(n_j+1)/m_j=\lambda$ for every $j$ such that
\[
 X\simeq\prod_{j=1}^k\mathbb P^{n_j},\qquad
 \varphi^*\mathcal O_{\mathbb P^s}(1)\simeq
 \boxtimes_{j=1}^k\mathcal O_{\mathbb P^{n_j}}(m_j).
\]
Moreover $s=N_{\mathbf m}$, where $N_{\mathbf m}+1=\prod_j\binom{n_j+m_j}{n_j}$,
and there are $F\in\Aut^0(X)$ and $U\in\operatorname{PU}(N_{\mathbf m}+1)$ such
that
\[
 \varphi=U\circ\iota_{\mathbf m}\circ F,
 \qquad
 \omega=F^*\Bigl(\bigoplus_{j=1}^k m_j\omega_{\FS,j}\Bigr),
\]
where $\iota_{\mathbf m}$ is the complete product Veronese--Segre immersion and
$\bigoplus_jm_j\omega_{\FS,j}$ denotes the product metric
$\sum_jm_j\,\mathrm{pr}_j^*\omega_{\FS}$.  Conversely, every such immersion
induces a K\"ahler--Einstein metric.
\end{mainthm}

The proof, given in Section~\ref{sec:global}, runs as follows.  After the
torus is adapted to the metric, a reduction of Di Scala and Sombra
\cite[Lemma~2.8]{DiScalaSombra2025} makes the immersion toric and its
polynomial GEC, so Theorem~\ref{thm:intro-surface} applies to every two-face
of the polarization polytope $P$.  A combinatorial
criterion of Wiemeler and Yu--Masuda \cite{Wiemeler2015,YuMasuda2021} and a
lattice-untwisting lemma then identify $P$ with $\prod_jm_j\Sigma_{n_j}$,
where $\Sigma_n=\conv(0,e_1,\ldots,e_n)$; Bando--Mabuchi uniqueness
\cite{BandoMabuchi1987} and Calabi rigidity \cite{Calabi1953} give the metric
and the immersion.  If $\omega$ is invariant under
the given torus, even a possibly sparse monomial immersion is forced to be the
complete product system (Corollary~\ref{cor:intro-fixed-torus}).

Theorem~\ref{thm:intro-global} proves the compact toric case of the
Loi--Zedda conjecture.  It also yields rigidity for projectively induced cscK
metrics, regular quantizations, and homothetic balanced metrics on the
positive first-Chern ray, and for toric K\"ahler--Ricci solitons in projective
space (Section~\ref{sec:global}).  Finally, combined with Hulin's completion
theorem \cite{Hulin1996} and the torus-extension step formulated by Manno and
Salis \cite{MannoSalis2024}, it gives the local classification that they
conjectured.

\begin{maincor}[Local and univalent $T^d$-invariant classification]
\label{cor:intro-manno-salis}
Let $(Y,\omega)$ be a connected K\"ahler $d$-fold with an effective
holomorphic Hamiltonian $T^d$-action that has a fixed point, and let
$\varphi\colon(Y,\omega)\to(\mathbb P^s,\omega_{\FS})$ be a full K\"ahler
immersion with $\Ric(\omega)=\lambda\omega$, $\lambda>0$.  Then the germ of
$(Y,\omega)$ at the fixed point is holomorphically isometric to a germ of
$\bigl(\prod_j\mathbb P^{n_j},\bigoplus_jm_j\omega_{\FS,j}\bigr)$ with
$\sum_jn_j=d$ and $(n_j+1)/m_j=\lambda$.  If $\varphi$ identifies $Y$ with an
embedded submanifold of $\mathbb P^s$ (univalence), then all of $(Y,\omega)$
is holomorphically isometric to an open subset of this product.
\end{maincor}

Since the Riemannian Einstein constant of Manno and Salis is twice ours, their
normalization $\Ric(g)=2g$ corresponds to $m_j=n_j+1$.
Corollary~\ref{cor:intro-manno-salis} therefore proves in all dimensions the
germ form of \cite[Conjecture~3]{MannoSalis2024} and the univalent form of
their Conjecture~2, previously known through dimension six
\cite{MannoSalis2026}.  Without univalence the global statement can fail
(Remark~\ref{rem:covering-obstruction}).  In both
Theorem~\ref{thm:intro-global} and Corollary~\ref{cor:intro-manno-salis}, the
new ingredient is Theorem~\ref{thm:intro-surface}: it replaces the
dimension-by-dimension enumeration of the earlier proofs by a single
two-dimensional argument.  The proofs also never rescale the Einstein
constant to one, a step that Di Scala and Sombra observe need not stay within
projective immersions (Remark~\ref{rem:normalization-caveat}); in particular,
the earlier classifications through dimension six follow from
Theorem~\ref{thm:intro-global} without that step.

\subsection{Finite-support exponential families}

We now drop the lattice.  For an arbitrary finite carrier $A\subset\R^d$, as
in statistics, there is no lattice, torus action, or compact manifold to work
with, and \eqref{eq:intro-fisher-determinant} is simply an identity for a
real-analytic function on $\R^d$.  The rigidity nevertheless persists.

\begin{mainthm}[Finite-support Fisher-determinant rigidity]
\label{thm:intro-nef}
Let $A$, $c_a$, $Z$, and $\psi$ be as above, and suppose that
\eqref{eq:intro-fisher-determinant} holds.  Then $\lambda>0$, and there are
positive integers $n_j,m_j$ $(1\le j\le k)$ with $\sum_jn_j=d$ and
$(n_j+1)/m_j=\lambda$, vectors $u_{jq}$ forming a basis of $\R^d$, a vector
$v\in\R^d$, and constants $K,\gamma_{jq}>0$ such that
\begin{equation}
 Z(\theta)=Ke^{\langle v,\theta\rangle}
 \prod_{j=1}^k
 \Bigl(1+\sum_{q=1}^{n_j}\gamma_{jq}e^{\langle u_{jq},\theta\rangle}\Bigr)^{m_j}.
 \label{eq:intro-nef-factorization}
\end{equation}
Conversely, every function of this form with $(n_j+1)/m_j=\lambda$ for all
$j$ satisfies \eqref{eq:intro-fisher-determinant} for suitable $C>0$ and
$b\in\R^d$.  Consequently, the Jeffreys prior of a minimal finite-support
natural exponential family is a proper Diaconis--Ylvisaker conjugate prior if
and only if, after an affine change of sufficient statistic, the family is a
product of multinomial families with $n_j+1$ categories and $m_j$ trials,
where $(n_j+1)/m_j$ is independent of $j$.
\end{mainthm}

For $d\ge2$, the main new step is to manufacture the missing torus
(Lemma~\ref{lem:intrinsic-phase-torus}).  By Hulin's completion theorem, a
germ of the image of $\Phi$ extends to a compact K\"ahler--Einstein
submanifold of projective space.  The closure of the unitary phase action
$\Phi(z)\mapsto\Phi(z+it)$, $t\in\R^d$, is then an effective Hamiltonian torus
of dimension exactly $d$, and the differences of points of $A$ generate its
character lattice, so Corollary~\ref{cor:intro-fixed-torus} applies.

For statisticians, Corollary~\ref{cor:nef-statistics-vocabulary} restates
Theorem~\ref{thm:intro-nef} in the language of natural exponential families:
the conjugate families of Consonni and Veronese, the generalized variance and
the skewness vector, and the determinant measure of Kokonendji and Seshadri
\cite{KokonendjiSeshadri1996}.  In Casalis's terms, among families generated
by finitely supported measures the two conjugate families coincide exactly for
affine images of matched multinomial products.  Up to affinity and convolution
power there is one such type for each partition of $d$, and only the
one-block multinomial type is simple quadratic (Remark~\ref{rem:nef-types}).
Theorem~\ref{thm:intro-nef} thus settles the finite-support case of Casalis's
question; since the one-dimensional case is classical, what remains open is
the infinite-support case for $d\ge2$.

Two earlier lines of work approach the same question from other directions.
Hassairi reformulated the coincidence of the two conjugate families through
his generalized variance transform, which requires the generalized variance
to be a Laplace transform, as it is for infinitely divisible generating
measures \cite{Hassairi1999}.  As he notes, this fails for the multinomial
family; in fact it fails for every finite support, where the generalized
variance is bounded and nonconstant, and
Corollary~\ref{cor:nef-statistics-vocabulary}(5) supplies the finite-support
analogue.  Ghribi and Masmoudi characterized multinomial families by their
generalized variance under a bounded-support hypothesis
\cite{GhribiMasmoudi2010}.  Section~\ref{subsec:nef-literature} discusses
these and further antecedents.

\subsection{Organization}

Section~\ref{sec:preliminaries} recalls the GEC machinery of Di Scala and
Sombra, and Sections~\ref{sec:ledger}--\ref{sec:polygon} prove
Theorem~\ref{thm:intro-surface}.  Section~\ref{sec:global} proves
Theorem~\ref{thm:intro-global}, Corollary~\ref{cor:intro-fixed-torus}, and the
consequences for cscK, balanced, and soliton metrics.
Section~\ref{sec:finite-nef} proves Theorem~\ref{thm:intro-nef} and relates it
to the statistical literature, and Section~\ref{sec:local} proves
Corollary~\ref{cor:intro-manno-salis}.  Section~\ref{sec:scope} compares our
results with linear precision and with the work of Manno and Salis, and it
closes with limitations and open problems.

The sections need not be read in order.  Readers mainly interested in the
statistical results can go directly to Section~\ref{sec:finite-nef}, which
uses the earlier sections only through Corollary~\ref{cor:intro-fixed-torus}.
Readers mainly interested in the geometry can read
Section~\ref{sec:preliminaries} and then Section~\ref{sec:global}, taking
Theorem~\ref{thm:intro-surface} as given.

\section{Laurent Monge--Amp\`ere polynomials and GEC}
\label{sec:preliminaries}

This section fixes notation for the logarithmic Monge--Amp\`ere polynomial and
collects the results of Di Scala and Sombra used in the rest of the paper.
Apart from the coordinate-free presentation, adapted to faces of Newton
polytopes, the material is theirs.

Let $M\simeq\Z^d$ be a lattice, write $\C[M]$ for its Laurent group algebra,
and denote its characters by $\chi^a$, $a\in M$.  Write a Laurent polynomial
as $f=\sum_{a\in\mathcal A}c_a\chi^a$ with all $c_a\ne0$, so that
$\mathcal A=\supp(f)$.  We use the calligraphic letter $\mathcal A$ for such
lattice exponent supports, to distinguish them from the real carrier
$A\subset\R^d$ of the introduction and Section~\ref{sec:finite-nef}.  Put
\[
 M_f=\Span_\R(\mathcal A-\mathcal A)\cap M,
 \qquad d_f=\operatorname{rank}M_f.
\]
Choose $m_0\in\supp(f)$ and put $\widehat f=\chi^{-m_0}f\in\C[M_f]$.
After choosing a basis of $M_f$, set
\begin{equation}
 \delta(f)=\det\Theta^2\log\widehat f,
 \qquad
 \mu(f)=\widehat f^{d_f+1}\delta(f),
 \qquad \Theta_i=x_i\partial_{x_i}.
 \label{eq:def-mu}
\end{equation}
A basis change lies in $\operatorname{GL}(M_f)$, so the logarithmic Hessian
transforms by congruence and its determinant is unchanged because the basis
matrix has determinant $\pm1$.  A different choice of $m_0$, or multiplication
of $f$ by a scalar and a character, changes $\mu$ only by the corresponding
nonzero Laurent unit.  These facts, together with the power and product
rules for $\mu$, are Definition~3.5 and Lemma~3.8 of
\cite{DiScalaSombra2025}.  We always compute \eqref{eq:def-mu} in the
effective difference lattice.  In particular a face polynomial uses the
dimension of that face, not the dimension of the ambient polytope.  When the
support has already been translated into $M_f$, we suppress the hat.

The Cauchy--Binet expansion shows directly that $\mu(f)$ is a Laurent
polynomial.  If $d_f=d$ and the columns $\widetilde a=(1,a)$ are written in
a lattice basis, then, up to the Laurent unit caused by the chosen support
translation,
\begin{equation}
 \mu(f)=
 \sum_{\substack{B\subset\mathcal A\\ |B|=d+1}}
 \det(\widetilde a:a\in B)^2
 \prod_{a\in B}c_a\chi^a.
 \label{eq:cauchy-binet}
\end{equation}
The measure-theoretic determinant identity underlying
\eqref{eq:cauchy-binet} goes back to
Kokonendji--Seshadri \cite[Theorem~2.2]{KokonendjiSeshadri1996};
\eqref{eq:cauchy-binet} is its finite positive-atomic form, extended
algebraically to arbitrary coefficients.  The Laurent-lattice formulation
and the Newton-polytope consequences used below are Lemma~3.6 and
Theorem~3.10 of \cite{DiScalaSombra2025}.  For positive coefficients no
cancellation occurs in \eqref{eq:cauchy-binet}.

\begin{definition}[{\cite[Definition~3.18]{DiScalaSombra2025}}]
A Laurent polynomial $f$ whose support is unimodular in $M_f$ satisfies the
\emph{generalized Einstein condition} (GEC) if
\begin{equation}
 \mu(f)\mid \widehat f^{\,N}
 \label{eq:def-gec}
\end{equation}
in $\C[M_f]$ for some integer $N>0$.
\end{definition}

Besides the definition, we use three results of Di Scala and Sombra.  First,
GEC is hereditary:
if $f$ has unimodular support and satisfies \eqref{eq:def-gec}, the translated
initial polynomial on every face satisfies GEC in its effective lattice
\cite[Corollary~3.15 and Proposition~3.19]{DiScalaSombra2025}.  Second, a rank-one unimodular GEC
polynomial is, up to a Laurent unit, a power of a binomial
\cite[Proposition~4.1]{DiScalaSombra2025}.  Third, their rank-two initial-part
formula gives the following edge data
\cite[Theorem~3.14, Corollary~3.15, and Proposition~4.3]{DiScalaSombra2025}.

\begin{proposition}[Rank-two edge input]
\label{prop:edge-input}
Let $f$ have effective rank two and unimodular support, and satisfy GEC.  Let $E$ be an
edge of $P=\NP(f)$, of lattice length $L$, and let $E'$ be the nearest inner
lattice-parallel slice, of length $L'$; length zero is allowed.  In a
primitive coordinate $t$ along $E$, the restrictions to $E$ and $E'$ are,
up to Laurent units and nonzero constants,
\begin{equation}
 f|_E=(t+\xi)^L,
 \qquad f|_{E'}=(t+\xi)^{L'}.
 \label{eq:edge-rows}
\end{equation}
Moreover, $L'/L$ is independent of $E$, and the face-initial factorization of
$\mu(f)$ implies
\begin{equation}
 \ell_E\bigl(\NP(\mu(f))\bigr)=2L-2+L',
 \label{eq:mu-edge-length}
\end{equation}
where $\ell_E(Q)$ denotes the lattice length of the face of a lattice polygon
$Q$ exposed by the inner normal of $E$.
\end{proposition}

For completeness, the term $2L-2$ in
\eqref{eq:mu-edge-length} is the Newton length of the rank-one
Monge--Amp\`ere polynomial of $(t+\xi)^L$.  The conditions in
Proposition~\ref{prop:edge-input} are necessary but not sufficient; the rest
of the proof supplies the missing global factorization constraint.

\section{The edge ledger}
\label{sec:ledger}

The proof of Theorem~\ref{thm:intro-surface} begins with a bookkeeping
identity.  It compares, edge by edge, the Newton polygon of $p$ with that of
$\mu(p)$ and records the difference in terms of the irreducible factors of
$p$.  Throughout this section and the next two, $p\in\C[M]$ satisfies the
hypotheses of Theorem~\ref{thm:intro-surface}.  Factor $p$ in the Laurent
UFD:
\begin{equation}
 p=c\chi^u\prod_{\nu=1}^t f_\nu^{m_\nu},
 \qquad m_\nu>0,
 \label{eq:p-factor}
\end{equation}
where the $f_\nu$ are pairwise nonassociate irreducibles.
Because $\mu(p)\mid p^N$ and $p^N\ne0$, we have $\mu(p)\ne0$.
GEC says that every nonunit irreducible factor of $\mu(p)$ is among the
$f_\nu$, and hence
\begin{equation}
 \mu(p)=c'\chi^v\prod_{\nu=1}^t f_\nu^{e_\nu},
 \qquad e_\nu\ge0.
 \label{eq:mu-factor}
\end{equation}
Define the integer defect
\begin{equation}
 r_\nu=3m_\nu-e_\nu.
 \label{eq:defect}
\end{equation}
Note that $r_\nu$ may, a priori, have either sign.

For an edge $E$ of $P=\NP(p)$ and a lattice polygon $Q$ whose normal fan is
refined by that of $P$, let $\ell_E(Q)$ be, as in
Proposition~\ref{prop:edge-input}, the lattice length of the face of $Q$
exposed by the inner normal of $E$.  Newton polytopes add under
multiplication.  Applying exposed-face lengths to
\eqref{eq:p-factor}--\eqref{eq:mu-factor} gives, with $L=\ell_E(P)$ the lattice
length of $E$,
\begin{equation}
 3L-\ell_E\bigl(\NP(\mu(p))\bigr)
 =\sum_\nu r_\nu\ell_E\bigl(\NP(f_\nu)\bigr).
 \label{eq:ledger-pre}
\end{equation}

Let $C_E$ denote the invariant curve of the smooth toric surface $X_P$
corresponding to $E$.  The standard wall relation gives
\begin{equation}
 L'=L-C_E^2.
 \label{eq:inner-selfint}
\end{equation}
To fix the sign convention, put $E=[(0,0),(L,0)]$ with $P$ locally above
the $x$-axis.
If the primitive outgoing directions of the two neighboring edges are
$(\alpha,1)$ and $(\beta,1)$, then $L'=L+\beta-\alpha$, while the adjacent
normal relation gives $C_E^2=\alpha-\beta$.  This proves
\eqref{eq:inner-selfint} with our sign convention.  Substituting
Proposition~\ref{prop:edge-input} into \eqref{eq:ledger-pre} yields our basic
identity.

\begin{proposition}[Signed factor-valuation ledger]
\label{prop:ledger}
For every edge $E$ of $P$,
\begin{equation}
 \boxed{
 2+C_E^2=\sum_\nu r_\nu
 \ell_E\bigl(\NP(f_\nu)\bigr)}.
 \label{eq:ledger}
\end{equation}
\end{proposition}

Equation~\eqref{eq:ledger} is an equality of integers, or equivalently an
edgewise equality in the Grothendieck group of lattice polygons.  Since the
$r_\nu$ may be negative, its right side is not a Minkowski decomposition by
nef polygons.

\section{Valuations along irreducible factors}
\label{sec:valuations}

The ledger becomes useful once the defects $r_\nu$ are known.  We now compute
them for the two kinds of factors that matter: a two-dimensional irreducible
factor always has defect three, and a binomial factor has defect two provided
the rest of $p$ meets it transversely.  An example after
Lemma~\ref{lem:binomial-factor} shows that transversality cannot be dropped.

\begin{lemma}[Two-dimensional factors]
\label{lem:full-factor}
Let $f\in\C[M]$ be irreducible with $\dim\NP(f)=2$.  Then
\begin{equation}
 f\nmid\mu(f).
 \label{eq:f-not-mu}
\end{equation}
Consequently, if $p=f^m g$ with $f\nmid g$, then
\begin{equation}
 \ord_f\delta(p)=-3,
 \qquad \ord_f\mu(p)=3m-3,
 \qquad r_f=3.
 \label{eq:full-defect}
\end{equation}
\end{lemma}

\begin{proof}
Set
\[
 A=\Theta_1f,\quad B=\Theta_2f,\quad C=\Theta_1^2f,
 \quad H=\Theta_1\Theta_2f,\quad F=\Theta_2^2f.
\]
Expanding $f^3\det \Theta^2\log f$ and reducing modulo $f$ gives
\begin{equation}
 \mu(f)\equiv-\bigl(CB^2+FA^2-2HAB\bigr)\pmod f.
 \label{eq:mu-mod-f}
\end{equation}
On the smooth locus of $f=0$, the logarithmic vector field
$T=B\Theta_1-A\Theta_2$ is tangent to the curve, and on the patch $B\ne0$,
\begin{equation}
 T(A/B)=\frac{CB^2+FA^2-2HAB}{B^2}.
 \label{eq:gauss-derivative}
\end{equation}
If $f\mid\mu(f)$, the logarithmic Gauss map
$[\Theta_1f:\Theta_2f]$ is therefore constant on the normalization of the irreducible
curve.  Hence, for some $(\alpha,\beta)\ne(0,0)$,
\begin{equation}
 f\mid \alpha\Theta_1f+\beta\Theta_2f.
 \label{eq:derivative-divisibility}
\end{equation}
If the polynomial on the right vanishes, coefficient comparison immediately
puts $\supp(f)$ on an affine line.  Otherwise its Newton polygon is contained
in $\NP(f)$.  Divisibility and Newton-polytope additivity force the quotient
in \eqref{eq:derivative-divisibility} to be a Laurent monomial; boundedness of
$\NP(f)$ forces its translation vector to vanish.  Coefficient comparison
again says that $\alpha a_1+\beta a_2$ is constant on $\supp(f)$, contradicting
$\dim\NP(f)=2$.  This proves \eqref{eq:f-not-mu}.

Now write
\[
 \Theta^2\log p=m\Theta^2\log f+\Theta^2\log g.
\]
The determinant of the first summand is $m^2\delta(f)$ and has an exact pole
of order three along $f=0$ by \eqref{eq:f-not-mu}.  A mixed determinant using
one entry of $\Theta^2\log g$, which is regular at the generic point of $f=0$,
has pole order at most two.  The leading pole cannot cancel.  Multiplication
by $p^3$ proves \eqref{eq:full-defect}.
\end{proof}

\begin{lemma}[Transverse binomial factors]
\label{lem:binomial-factor}
Let $h$ be an irreducible Laurent binomial and write $p=h^m g$, with
$h\nmid g$.  Let $\bar g$ be the restriction of $g$ to the translated
one-dimensional subtorus $h=0$.  If $\bar g$ is not a Laurent monomial, then
\begin{equation}
 \ord_h\delta(p)=-2,
 \qquad \ord_h\mu(p)=3m-2,
 \qquad r_h=2.
 \label{eq:binomial-defect}
\end{equation}
\end{lemma}

\begin{proof}
The exponent difference of an irreducible Laurent binomial is primitive.
After an integral monomial coordinate change and multiplication by a unit,
take $h=1+\alpha x$.  Then
\[
 \Theta^2\log h=
 \begin{pmatrix}
  \alpha x/(1+\alpha x)^2&0\\0&0
 \end{pmatrix}.
\]
Put $B=\Theta^2\log g$.  At the generic point of $h=0$, $B$ is regular and
\begin{equation}
 \det\bigl(m\Theta^2\log h+B\bigr)
 =\det B+m\frac{\alpha x}{(1+\alpha x)^2}B_{22}.
 \label{eq:binomial-determinant}
\end{equation}
The restriction of $B_{22}$ is $\Theta_y^2\log\bar g$.  It vanishes identically
exactly when $\bar g=cy^q$.  Indeed,
\[
 \Theta_y(\Theta_y\log\bar g)=0
\]
makes $y\bar g'/\bar g$ constant, and comparison of Laurent coefficients
gives the claim.  By hypothesis this does not occur,
so \eqref{eq:binomial-determinant} has an exact double pole along $h=0$.
Multiplying by $p^3$ proves \eqref{eq:binomial-defect}.
\end{proof}

The hypothesis in Lemma~\ref{lem:binomial-factor} cannot be dropped.  For
example, $p=(1+x)(1+x+y)$ has complementary factor $1+x+y$, which restricts
to the monomial $y$ on $x=-1$, and $\mu(p)=xy(1+x)^2(2+2x+y)$, so the defect
of $1+x$ is one.  The transverse hypothesis is verified where the lemma is
used, in the proofs of Propositions~\ref{prop:six-exclusion}
and~\ref{prop:rectangle-recovery}.

\begin{lemma}[Directions and roots]
\label{lem:directions}
If an irreducible binomial $h$ divides $p$, its Newton segment is parallel to
an edge of $P$.  The normal fan of $P$ contains the pair of opposite rays
normal to that segment.  For each fixed unoriented edge direction, all
irreducible binomial factors of $p$ are associates.
\end{lemma}

\begin{proof}
The normal fan of a Minkowski sum refines the normal fan of each summand.  A
segment has a wall consisting of the two opposite normal rays, so those rays
must occur in the normal fan of $P$.  Expose \eqref{eq:p-factor} by either
normal perpendicular to the segment.  The whole binomial, rather than just
one monomial, appears in the initial product.  By
\eqref{eq:edge-rows}, that initial product is a power of one irreducible
one-variable binomial.  Unique factorization forces every binomial in the
same direction to have the same root, hence to be associate.
\end{proof}

For use in the next section, we package the ledger edge by edge.  For each
edge $E$, define
\begin{equation}
 A_E=\sum_{\dim\NP(f_\nu)=2}\ell_E\bigl(\NP(f_\nu)\bigr)\ge0,
 \qquad
 r_E=
 \begin{cases}
  r_h,&\text{if a binomial factor $h$ is parallel to $E$},\\
  0,&\text{if no such factor exists}.
 \end{cases}
 \label{eq:edge-defect-data}
\end{equation}
The second quantity is well-defined by Lemma~\ref{lem:directions}.  The
ledger and Lemma~\ref{lem:full-factor} therefore give, for every edge,
\begin{equation}
 2+C_E^2=3A_E+r_E.
 \label{eq:edge-defect-ledger}
\end{equation}

\section{Classification of the Newton polygon and the polynomial}
\label{sec:polygon}

We now combine the ledger with the defect computations.  The argument has
three steps: a common-sign relation for the self-intersections of the
boundary curves, the exclusion of polygons with six or more sides, and the
recovery of the factors for the two shapes that survive.

Write the invariant boundary curves of $X_P$ cyclically as
$C_1,\ldots,C_\sigma$, and let $L_i$ be the lattice length of the edge of $P$
corresponding to $C_i$.  Proposition~\ref{prop:edge-input} and
\eqref{eq:inner-selfint} imply
\begin{equation}
 C_i^2=(1-\rho)L_i
 \label{eq:common-sign}
\end{equation}
for a number $\rho$ independent of $i$.  Thus all $C_i^2$ have the same
strict sign unless all vanish.  We also need the following elementary fan
fact.

\begin{lemma}[Minus-one curves]
\label{lem:minus-one}
Let $X$ be a smooth complete toric surface all of whose invariant boundary
curves have negative self-intersection.  Then at least three of them have
self-intersection $-1$.
\end{lemma}

\begin{proof}
Let $u_1,\ldots,u_\sigma$ be the cyclic primitive fan generators, oriented so that
$\det(u_i,u_{i+1})=1$, and write $a_i=-C_i^2>0$.  The wall relation is
\begin{equation}
 u_{i-1}+u_{i+1}=a_i u_i.
 \label{eq:wall-relation}
\end{equation}
The origin lies in the interior of
$Q=\conv\{u_1,\ldots,u_\sigma\}$.  If $a_i\ge2$, then
\[
 u_i=\frac1{a_i}u_{i-1}+\frac1{a_i}u_{i+1}
 +\left(1-\frac2{a_i}\right)0,
\]
so $u_i$ is not a vertex of $Q$.  The polygon $Q$ has at least three
vertices, and every corresponding index must have $a_i=1$.
\end{proof}

\begin{proposition}[Exclusion of six or more sides]
\label{prop:six-exclusion}
A Newton polygon under the hypotheses of
Theorem~\ref{thm:intro-surface} has fewer than six sides.
\end{proposition}

\begin{proof}
For a smooth complete toric surface with $\sigma$ invariant boundary curves,
\begin{equation}
 \sum_{i=1}^\sigma C_i^2=12-3\sigma.
 \label{eq:noether}
\end{equation}
Indeed $c_2(X_P)=\sigma$ and $\chi(\mathcal O_{X_P})=1$, so Noether's formula gives
$K_{X_P}^2=12-\sigma$; expanding
$K_{X_P}^2=(\sum_iC_i)^2$ gives \eqref{eq:noether}.

Suppose $\sigma\ge6$.  The right side of \eqref{eq:noether} is negative, and the
common-sign relation \eqref{eq:common-sign} forces every $C_i^2<0$.
Lemma~\ref{lem:minus-one} supplies at least three edges $E$ with
$C_E^2=-1$.  At such an edge the ledger reads
\begin{equation}
 1=3A_E+r_E.
 \label{eq:minus-one-ledger}
\end{equation}
Here the notation is defined in \eqref{eq:edge-defect-data}; in particular,
Lemma~\ref{lem:directions} shows that no other binomial contributes to this
edge length.

Modulo three, \eqref{eq:minus-one-ledger} forces a binomial factor parallel
to every $(-1)$-edge.  A convex polygon has at most two edges in one
unoriented parallel class, so among the three forced factors choose
nonparallel binomials $h$ and $h'$.  Write $p=h^m g$ with $m$ maximal.  The
factor $g$ contains $h'$.  On $h=0$, the nonparallel binomial $h'$ restricts
to a nonunit Laurent polynomial.  Every remaining factor restricts
nontrivially because $h\nmid g$, and in the one-variable Laurent UFD a
product is a unit only when all factors are units.  Thus $g|_{h=0}$ is not a
Laurent monomial.  Lemma~\ref{lem:binomial-factor} gives $r_h=2$.
Returning to the $(-1)$-edge parallel to $h$, equation
\eqref{eq:minus-one-ledger} becomes
\[
 1=3A_E+2,
\]
which is impossible.
\end{proof}

\begin{proposition}[Triangle factor recovery]
\label{prop:triangle-recovery}
If $P=m\Sigma_2$ under the hypotheses of
Theorem~\ref{thm:intro-surface}, then, up to a Laurent unit,
$p=(\alpha_0+\alpha_1x+\alpha_2y)^m$ with all $\alpha_i\ne0$.
\end{proposition}

\begin{proof}
A segment cannot be a Minkowski summand of a triangle, because its normal fan
requires a pair of opposite rays.  Every nonunit irreducible factor of $p$ is
therefore two-dimensional and has defect three.  Since $C_E^2=1$, the ledger
gives at each edge
\begin{equation}
 3=3\sum_\nu\ell_E(\NP(f_\nu)).
 \label{eq:triangle-ledger}
\end{equation}
A full-dimensional normal fan coarsened by a complete three-ray fan must be
that same fan, so every factor contributes a positive integer at every edge.
Equation~\eqref{eq:triangle-ledger} forces one distinct irreducible factor,
whose three edge lengths are one.  It is a three-term simplex polynomial
$f=\alpha_0+\alpha_1x+\alpha_2y$, and Newton-polytope equality yields
$p=c\chi^u f^m$.
\end{proof}

\begin{proposition}[Parallelogram factor recovery]
\label{prop:rectangle-recovery}
If $P$ is a lattice parallelogram under the hypotheses of
Theorem~\ref{thm:intro-surface}, then, up to a Laurent unit and an integral
monomial change of coordinates,
$p=(\alpha_0+\alpha_1x)^r(\beta_0+\beta_1y)^s$ with $r,s>0$ and all displayed
coefficients nonzero.
\end{proposition}

\begin{proof}
At an edge in either of the two directions, $C_E^2=0$ and the ledger becomes
\begin{equation}
 2=3A_E+r_E.
 \label{eq:rectangle-ledger}
\end{equation}
Reduction modulo three forces a binomial in each direction.  The complement
of either contains the binomial in the other direction, so
Lemma~\ref{lem:binomial-factor} makes both defects equal to two.  Equation
\eqref{eq:rectangle-ledger} now gives $A_E=0$ on every edge, excluding all
two-dimensional irreducible factors.  Lemma~\ref{lem:directions} leaves one
associate binomial class in each independent direction, and the two side
lengths determine their multiplicities.  This gives the stated form.
\end{proof}

We can now finish the classification.

\begin{proof}[Proof of Theorem~\ref{thm:intro-surface}]
Every full-dimensional polygon has $\sigma\ge3$.
Proposition~\ref{prop:six-exclusion} leaves $\sigma=3,4,5$.  If $\sigma=3$,
equations
\eqref{eq:common-sign} and \eqref{eq:noether} make all three
self-intersections positive, and their sum is $3$; hence each equals $1$.
The fan is that of $\mathbb P^2$, and \eqref{eq:common-sign} makes all three
edge lengths equal.  Therefore, after an integral monomial change of
coordinates, $P$ is a translate of $m\Sigma_2$.

If $\sigma=4$, the sum in \eqref{eq:noether} is zero.  The common-sign alternative
forces $C_i^2=0$ for all $i$, so the fan is the product fan and $P$ is a
lattice parallelogram.  If $\sigma=5$, the sum is $-3$; five negative nonzero
integers would sum to at most $-5$, while the common zero or positive
alternatives are also impossible.  Thus five sides do not occur.

The polynomial normal forms now follow from
Propositions~\ref{prop:triangle-recovery}
and~\ref{prop:rectangle-recovery}.

Conversely, for a three-term unimodular simplex polynomial $f$, formula
\eqref{eq:cauchy-binet} makes $\mu(f)$ a nonzero Laurent monomial.  Since
$\Theta^2\log f^m=m\Theta^2\log f$,
\begin{equation}
 \mu(f^m)=m^2f^{3m-3}\mu(f),
 \label{eq:simplex-mu}
\end{equation}
which divides $(f^m)^3$.  For independent primitive binomials $h_1,h_2$, a
direct determinant calculation gives
\begin{equation}
 \mu(h_1^rh_2^s)=c\chi^q h_1^{3r-2}h_2^{3s-2}
 \label{eq:rectangle-mu}
\end{equation}
for a nonzero $c\chi^q$.  This also divides $p^3$, proving the converse and
the exponent-three assertion.
\end{proof}

\begin{remark}[Why blanket pole counting fails]
The transverse qualifier in Lemma~\ref{lem:binomial-factor} explains why a
naive derivative-order proof cannot replace the classification: there is no
blanket rule assigning defect two to every binomial factor.  For
example, positive polynomials with full lattice-point support on a Delzant
trapezoid, of the form
$(1+x)(y+(1+x)^k)$ can have $(1+x)$-multiplicity $k+1$ in $\mu$, larger than
three when $k\ge3$.  An additional coprime factor prevents GEC.
\end{remark}

\section{The all-dimensional compact toric application}
\label{sec:global}

We now prove Theorem~\ref{thm:intro-global}.  The main difficulty is that
neither the immersion nor the metric is assumed to respect the torus of the
given toric structure.  We therefore begin by conjugating that torus into the
isometry group of the metric, which turns the problem into one about
polytopes (Section~\ref{subsec:adapted-torus}).  We then identify the
polytope (Section~\ref{subsec:integral-product}) and, from it, the metric and
the immersion (Section~\ref{subsec:metric-rigidity}).  The reduction to a
toric immersion already appears in \cite[Lemma~2.8]{DiScalaSombra2025}; we
record the version for the chosen torus because the later polytope and
coefficient bookkeeping needs it.  Section~\ref{subsec:fixed-torus} sharpens
the conclusion when the metric is invariant under the original torus, and
Section~\ref{subsec:consequences} derives consequences for cscK, balanced,
and soliton metrics.

\subsection{Adapted-torus and two-face reduction}
\label{subsec:adapted-torus}

In this first step we replace the given torus by a conjugate one that
preserves the metric, and we show that the resulting immersion polynomial
satisfies GEC.  Put $L=\varphi^*\mathcal O_{\mathbb P^s}(1)$.
By Hulin's positivity theorem for compact projectively induced
K\"ahler--Einstein manifolds, $\lambda>0$ \cite{Hulin2000}.
The cohomological Einstein
identity $c_1(X)=\lambda c_1(L)$ makes $X$ Fano.

\begin{lemma}[Adapted torus]
\label{lem:adapted-torus}
Let $T_c$ be the compact torus of the chosen toric structure.  There is an
$A\in\Aut^0(X)$ for which $A^*\omega$ is $T_c$-invariant.  After a unitary
change of target coordinates, the full immersion
$\varphi_A=\varphi\circ A$ is toric for $T_c$ and has the form
\begin{equation}
 \varphi_A(x)=(\alpha_0\chi^{a_0}(x):\cdots:
                  \alpha_s\chi^{a_s}(x)),
 \qquad
 p=p_{\varphi_A}=\sum_{j=0}^s|\alpha_j|^2\chi^{a_j}.
 \label{eq:toric-immersion-polynomial}
\end{equation}
The exponents $a_j$ are distinct, the coefficients are positive, the support
$S=\{a_0,\ldots,a_s\}$ is unimodular, and $P=\NP(p)$ is the Delzant polytope
of $L$ up to translation.
\end{lemma}

\begin{proof}
Matsushima's theorem makes $\Aut^0(X)$ reductive in the K\"ahler--Einstein
case \cite{Matsushima1957}.  Let
$K_\omega=\operatorname{Isom}^0(X,\omega)\subset\Aut^0(X)$.  Calabi's
maximal-compact theorem for constant-scalar-curvature K\"ahler metrics says
that $K_\omega$ is a maximal compact subgroup of $\Aut^0(X)$
\cite[Theorem~3]{Calabi1985}.  The compact group $T_c$ lies in a maximal
compact subgroup, and maximal compact subgroups are conjugate.
Hence some $A\in\Aut^0(X)$ satisfies $AT_cA^{-1}\subset K_\omega$.  It
follows that $A^*\omega$ is $T_c$-invariant and is induced by the full
immersion $\varphi\circ A$.

Lemma~2.7 of Di Scala--Sombra makes this immersion toric after a unitary
target change, and their Lemma~2.6 gives all the asserted support properties
\cite{DiScalaSombra2025}.  Finally, $A^*L\simeq L$: on a Fano manifold
$\operatorname{Pic}^0(X)=0$, and the connected group $\Aut^0(X)$ acts
trivially on the discrete Picard group.  Thus $P$ is the polytope of $L$ for
the chosen torus, up to its usual lattice translation.
\end{proof}

The Einstein polynomial identity of Corollary~2.13 and equation~(3.5) of
\cite{DiScalaSombra2025}, applied to $A^*\omega$, is
\begin{equation}
 \mu(p)=c\chi^u p^{d+1-\lambda}.
 \label{eq:einstein-polynomial}
\end{equation}
Their equation~(3.6) and Proposition~3.16 say that
$q=\gamma p^\lambda$ is a Laurent polynomial.  Consequently
\begin{equation}
 p^{d+1}=(c\gamma)^{-1}\chi^{-u}q\,\mu(p),
 \label{eq:original-p-gec}
\end{equation}
so the positive immersion polynomial $p$ in
\eqref{eq:toric-immersion-polynomial} itself satisfies GEC.

\begin{remark}[Normalization caveat]
\label{rem:normalization-caveat}
It is common to rescale a K\"ahler--Einstein form so that its Einstein
constant becomes~$1$.  For a projectively induced metric this replaces $p$ by
$q=\gamma p^\lambda$ as in \cite[(3.6)]{DiScalaSombra2025}.  By
\cite[Proposition~3.16]{DiScalaSombra2025}, $q$ is again a Laurent
polynomial, but its coefficients need not be positive
\cite[Remark~3.17]{DiScalaSombra2025}, so $q$ need not come from a
projective immersion.  As Di Scala and Sombra point out
\cite[Section~3.5]{DiScalaSombra2025}, the rescaling therefore cannot be
assumed to take place within projective immersions, contrary to what might
be inferred from \cite[p.~485]{ArezzoLoiZuddas2012} and
\cite[Section~2.5.7]{MannoSalis2026}.  The proof of
Theorem~\ref{thm:intro-global} avoids this issue: it uses $q$ only as a
Laurent polynomial in \eqref{eq:original-p-gec}, never as an immersion
polynomial, and it proves GEC for the original positive $p$.  Since
Theorem~\ref{thm:intro-global} holds for every Einstein constant, it contains
the classification through complex dimension six of
\cite{ArezzoLoiZuddas2012,MannoSalis2026} without passing through the
rescaling.
\end{remark}

By Proposition~3.19 of \cite{DiScalaSombra2025}, GEC passes to every face
polynomial, so Theorem~\ref{thm:intro-surface} applies to each two-face of
$P$.  Hence every two-face of the polarization polytope is a dilated
unimodular triangle or a lattice parallelogram, even though the original
metric and immersion were not assumed to be toric for the chosen action.

\subsection{From two-faces to an integral product}
\label{subsec:integral-product}

We now know the shape of every two-face of $P$, and we use this information to
recover $P$ itself.  For $d\ge3$, Yu--Masuda state that a simple polytope all of whose two-faces
are triangles or quadrilaterals is combinatorially a product of simplices
\cite[Theorem~2.1]{YuMasuda2021}.  They attribute this criterion to Wiemeler;
compare \cite[Proposition~4.5]{Wiemeler2015}.  The cases
$d=1,2$ are immediate from the preceding classification.  A combinatorial
product need not be an integral-affine product, so the following extra step is
essential.

\begin{lemma}[Lattice untwisting]
\label{lem:lattice-untwisting}
Let $P$ be a Delzant lattice polytope combinatorially equivalent to
$\prod_{j=1}^k\Sigma_{n_j}$.  If every quadrilateral two-face is a lattice
parallelogram, then, up to integral-affine equivalence,
\begin{equation}
 P=\prod_{j=1}^k \ell_j\Sigma_{n_j}
 \label{eq:lattice-product}
\end{equation}
for positive integers $\ell_j$.
\end{lemma}

\begin{proof}
Label the vertices by tuples
\[
 \alpha=(\alpha_1,\ldots,\alpha_k),
 \qquad \alpha_j\in\{0,\ldots,n_j\},
\]
and let $v$ be the all-zero vertex.  Define
\[
 w_{jq}=v_{(0,\ldots,q,\ldots,0)}-v,
 \qquad w_{j0}=0.
\]
Every four-cycle obtained by changing coordinates in two distinct factors is
a quadrilateral two-face.  Its parallelogram identity makes the displacement
in either coordinate independent of the other coordinates.  Induction on the
number of nonzero coordinates yields
\begin{equation}
 v_\alpha=v+\sum_jw_{j,\alpha_j}.
 \label{eq:vertex-additivity}
\end{equation}

Write $w_{jq}=a_{jq}e_{jq}$ with $e_{jq}$ primitive.  The $d$ directions
$e_{jq}$ at $v$ form a lattice basis by the Delzant condition.  For distinct
$q,t$ in one block, the vertices $v$, $v+w_{jq}$, and $v+w_{jt}$ span a
triangular two-face.  If $g=\gcd(a_{jq},a_{jt})$, its third primitive edge
direction is
\[
 \frac{a_{jt}e_{jt}-a_{jq}e_{jq}}{g}.
\]
Smoothness at $v+w_{jq}$ forces $a_{jt}/g=1$, and smoothness at
$v+w_{jt}$ forces $a_{jq}/g=1$.  Hence all $a_{jq}$ in block $j$ have a
common value $\ell_j$.  Equation~\eqref{eq:vertex-additivity}, together with
the lattice basis at $v$, identifies the convex hull of all vertices with
\eqref{eq:lattice-product}.
\end{proof}

Apply Lemma~\ref{lem:lattice-untwisting} to $P$ and rename its side lengths
$m_j$.  Its normal fan is a product of simplex fans, and the
polytope--line-bundle dictionary gives
\begin{equation}
 X\simeq\prod_{j=1}^k\mathbb P^{n_j}
 \label{eq:variety-product}
\end{equation}
and
\begin{equation}
 P\simeq\prod_{j=1}^km_j\Sigma_{n_j},
 \qquad
 L\simeq\boxtimes_{j=1}^k\mathcal O_{\mathbb P^{n_j}}(m_j).
 \label{eq:polarized-product}
\end{equation}

\subsection{Metric and immersion rigidity}
\label{subsec:metric-rigidity}

With the variety and its polarization identified, it remains to determine the
Einstein constant, the metric, and the immersion.  Let $H_j$ be the
hyperplane class pulled back
from the $j$th factor.  With the convention $[\omega]/2\pi=c_1(L)$, the
cohomological Einstein equation is
\[
 c_1(X)=\lambda c_1(L),
\]
and hence
\[
 \sum_j(n_j+1)H_j=\lambda\sum_jm_jH_j.
\]
Therefore
\begin{equation}
 \lambda=\frac{n_j+1}{m_j}\quad\text{for every }j.
 \label{eq:cohomological-matching}
\end{equation}

Set
\begin{equation}
 \omega_0=\bigoplus_jm_j\omega_{\FS,j}.
 \label{eq:product-metric}
\end{equation}
It lies in the class of $\omega$ and, by
\eqref{eq:cohomological-matching}, has the same Einstein constant.  After
rescaling both forms into $2\pi c_1(X)$, Bando--Mabuchi uniqueness
\cite[Theorem~A(ii)]{BandoMabuchi1987} gives $F\in\Aut^0(X)$ such that
\begin{equation}
 \omega=F^*\omega_0.
 \label{eq:bm}
\end{equation}

Let $\iota_{\mathbf m}\colon X\to\mathbb P^{N_{\mathbf m}}$ be the complete product
Veronese--Segre immersion associated to \eqref{eq:polarized-product}, with
the standard multinomial normalization of its coordinates.  Then
$\iota_{\mathbf m}^*\omega_{\FS}=\omega_0$, so
$\iota_{\mathbf m}\circ F$ and $\varphi$ are full K\"ahler immersions of the
same connected K\"ahler manifold.  Calabi rigidity
\cite[Theorem~9]{Calabi1953} gives $s=N_{\mathbf m}$ and a unitary projective
transformation $U\in\operatorname{PU}(N_{\mathbf m}+1)$ such that
\[
 \varphi=U\circ\iota_{\mathbf m}\circ F.
\]
This proves Theorem~\ref{thm:intro-global}; its converse follows from
$\Ric(m_j\omega_{\FS,j})=(n_j+1)\omega_{\FS,j}$ and
\eqref{eq:cohomological-matching}.

\subsection{Fixed-torus coefficient rigidity}
\label{subsec:fixed-torus}

Theorem~\ref{thm:intro-global} identifies the immersion only up to an
automorphism of $X$, because the metric need not be invariant under the torus
we started with.  When it is invariant, that automorphism can be absorbed into
the torus, and we obtain a statement about the coefficients of the immersion
polynomial itself.  This is the form used in Section~\ref{sec:finite-nef}.

\begin{corollary}[Fixed-torus coefficient rigidity]
\label{cor:intro-fixed-torus}
In Theorem~\ref{thm:intro-global}, suppose in addition that $\omega$ is
invariant under the chosen compact torus of $X$.  A unitary target change then
makes $\varphi$ toric.  If $p_\varphi=\sum_{a\in S}w_ax^a$ is its positive
Laurent polynomial, then, after an integral-affine relabeling of the
exponents,
\[
 \NP(p_\varphi)\simeq\prod_{j=1}^k m_j\Sigma_{n_j},
 \qquad
 p_\varphi(x)=C x^{\tau}\prod_{j=1}^k
 \Bigl(1+\sum_{q=1}^{n_j}c_{jq}x_{jq}\Bigr)^{m_j}
\]
for some $C,c_{jq}>0$ and $\tau\in M$.  Consequently $S$ is the full
lattice-point set of its Newton polytope: a possibly sparse toric immersion is
forced to be the complete product Veronese--Segre monomial system.
\end{corollary}

We now prove the corollary, so assume that $\omega$ is invariant under the
chosen compact torus $T_c$.
In Lemma~\ref{lem:adapted-torus} we may take $A=\id$.  Hence the polynomial
$p$ in \eqref{eq:toric-immersion-polynomial} is the immersion polynomial of $\varphi$
for the original action, after a unitary target change.  It is allowed to be
sparse at this stage; equations \eqref{eq:original-p-gec} and
\eqref{eq:polarized-product} have already proved GEC and the product form of
its Newton polytope.  We now recover every coefficient and thereby rule out
all omitted lattice points.

\begin{lemma}[Product automorphisms and torus normalizers]
\label{lem:product-normalizer}
Set $X=\prod_j\mathbb P^{n_j}$ and
$\omega_0=\bigoplus_jm_j\omega_{\FS,j}$.  Then
\begin{equation}
 \Aut^0(X)=\prod_j\operatorname{PGL}(n_j+1,\C),
 \qquad
 \operatorname{Isom}^0(X,\omega_0)=\prod_j\operatorname{PU}(n_j+1).
 \label{eq:product-groups}
\end{equation}
If $T_c$ is the standard compact torus and $T_\C$ its complexification, then
\begin{equation}
 N_{\Aut^0(X)}(T_c)=T_\C\rtimes\prod_j\mathfrak S_{n_j+1}.
 \label{eq:torus-normalizer}
\end{equation}
Every maximal torus of $\operatorname{Isom}^0(X,\omega_0)$ is conjugate to
$T_c$ inside that group.  Permutations of repeated projective factors, when
present in the full automorphism or isometry group, are discrete and amount
to allowed integral lattice relabelings.
\end{lemma}

\begin{proof}
The connected automorphism group preserves each extremal ray of the nef cone,
hence each factor projection, and acts on the factors by projective linear
maps.  This gives the first identity in \eqref{eq:product-groups}.  The
identity component of the holomorphic isometry group preserves the product
decomposition and restricts to the projective unitary group on each scaled
Fubini--Study factor, giving the second identity.  The usual conjugacy theorem
for maximal tori in a compact connected Lie group gives the last assertion.

In each $\operatorname{PGL}(n_j+1,\C)$ factor, an element normalizing the
standard diagonal torus permutes its one-dimensional weight spaces.  It is
therefore represented by a monomial matrix, uniquely a diagonal matrix times
a permutation matrix modulo scalars.  Conversely every such matrix normalizes
the torus.  Taking products proves \eqref{eq:torus-normalizer}.  Factor
permutations are outside the identity component and give only the stated
discrete relabelings.
\end{proof}

Let $F$ be as in \eqref{eq:bm}.  Because $\omega=F^*\omega_0$ is
$T_c$-invariant, $FT_cF^{-1}$ is a $d$-dimensional compact torus in
$\operatorname{Isom}^0(X,\omega_0)$ and hence is maximal.  By
Lemma~\ref{lem:product-normalizer}, an isometry $K$ of $\omega_0$ conjugates
it to $T_c$.  Replacing $F$ by $KF$ leaves \eqref{eq:bm} unchanged and makes
$F$ normalize $T_c$.  Equation~\eqref{eq:torus-normalizer} now writes $F$ as
a diagonal complex-torus element followed by permutation matrices.  The
permutation parts, together with any optional repeated-factor permutations,
give the allowed integral-affine relabeling.  After making that relabeling, the
diagonal part acts on the dense orbit by $x_{jq}\mapsto c_{jq}x_{jq}$ with
$c_{jq}>0$.

The standard product potential is the logarithm of
\[
 p_0(x)=\prod_j\left(1+\sum_{q=1}^{n_j}x_{jq}\right)^{m_j}.
\]
Thus $F^*\omega_0$ has torus-invariant potential
\[
 \log p_F(x),\qquad
 p_F(x)=\prod_j\left(1+\sum_{q=1}^{n_j}c_{jq}x_{jq}\right)^{m_j}.
\]
Equality of the two K\"ahler forms says that
$\log p-\log p_F$ has zero real Hessian in the logarithmic variables
$\theta$.  It is therefore a real affine function $a+\langle b,\theta\rangle$.
Exponentiating gives
$p(x)=e^a x^b p_F(x)$.  Both sides are finite exponential sums in $\theta$,
and their uniqueness says that one integral support is the translate by $b$
of the other.  Taking the difference of any matched pair of exponents forces
$b\in M$.  Consequently
\begin{equation}
 p(x)=C x^{\tau}\prod_j
 \left(1+\sum_{q=1}^{n_j}c_{jq}x_{jq}\right)^{m_j},
 \label{eq:coefficient-form}
\end{equation}
with $\tau=b\in M$ and $C=e^a>0$.  The multinomial theorem identifies all
coefficients of $p$ and shows, a posteriori, that
$\supp(p)=\tau+((\prod_jm_j\Sigma_{n_j})\cap M)$.  Thus the coordinates of
the toric full immersion in \eqref{eq:toric-immersion-polynomial} form a basis
of $H^0(X,L)$; after removing the phases of its coordinate coefficients by a
diagonal unitary target transformation, it is the product Veronese--Segre
system with the indicated positive torus translate.  Conversely, \eqref{eq:coefficient-form} is obtained
from \eqref{eq:product-metric} by the stated toric equivalences and is
K\"ahler--Einstein exactly when
\eqref{eq:cohomological-matching} holds.  This completes the proof of
Corollary~\ref{cor:intro-fixed-torus}.

\subsection{Balanced metrics, regular quantization, and solitons}
\label{subsec:consequences}

Theorem~\ref{thm:intro-global} has several immediate consequences once the
K\"ahler class is a positive multiple of the first Chern class.  They all rest
on the following elementary lemma, which also shows exactly where that
hypothesis enters.

\begin{lemma}[Constant scalar curvature on the first-Chern ray]
\label{lem:cscK-first-chern-ray}
Let $(X,\omega)$ be a compact connected K\"ahler manifold.  Suppose that
$\omega$ has constant scalar curvature and that
\begin{equation}
 c_1(X)=\lambda\frac{[\omega]}{2\pi}
 \label{eq:first-chern-ray}
\end{equation}
for some $\lambda\in\R$.  Then $\Ric(\omega)=\lambda\omega$.
\end{lemma}

\begin{proof}
The two closed real $(1,1)$-forms $\Ric(\omega)$ and $\lambda\omega$
represent the same cohomology class.  By the $\partial\bar\partial$-lemma,
\[
 \Ric(\omega)-\lambda\omega=i\partial\bar\partial f
\]
for a real smooth function $f$.  Contracting with $\omega$ shows that
$\Delta_\omega f$ is constant, because the scalar curvature is constant.
Its integral against the volume form is zero, so $\Delta_\omega f=0$.
Compactness then makes $f$ constant.
\end{proof}

\begin{corollary}[Positive-ray projectively induced cscK metrics]
\label{cor:cscK-positive-ray}
Let $X$ be a smooth compact toric manifold and let
$\varphi\colon X\to\mathbb P^s$ be a full holomorphic immersion.  Put
$\omega=\varphi^*\omega_{\FS}$ and
$L=\varphi^*\mathcal O_{\mathbb P^s}(1)$.  If $\omega$ has constant scalar
curvature and
\[
 c_1(X)=\lambda c_1(L),\qquad \lambda>0,
\]
then all the conclusions of Theorem~\ref{thm:intro-global} hold.  In
particular, $(X,\omega)$ is homogeneous.
\end{corollary}

\begin{proof}
Our normalization gives $[\omega]/2\pi=c_1(L)$, so
Lemma~\ref{lem:cscK-first-chern-ray} makes $\omega$ K\"ahler--Einstein.
Apply Theorem~\ref{thm:intro-global}.  Its product metric is homogeneous,
and pulling it back by a domain automorphism preserves homogeneity.
\end{proof}

\begin{corollary}[Regular quantization on the positive first-Chern ray]
\label{cor:regular-quantization}
Let $(X,L)$ be a smooth compact polarized toric manifold satisfying
$c_1(X)=\lambda c_1(L)$ for some $\lambda>0$.  Suppose that there is a
K\"ahler form $\omega\in2\pi c_1(L)$ for which $r\omega$ is balanced for
every sufficiently large positive integer $r$; that is, $\omega$ admits a
regular quantization.  Then, for positive integers
$n_j,m_j$,
\begin{equation}
 X\simeq\prod_j\mathbb P^{n_j},\qquad
 L\simeq\boxtimes_j\mathcal O_{\mathbb P^{n_j}}(m_j),\qquad
 \frac{n_j+1}{m_j}=\lambda,
 \label{eq:regular-product}
\end{equation}
and, for some $F\in\Aut^0(X)$,
\[
 \omega=F^*\!\left(\bigoplus_jm_j\omega_{\FS,j}\right).
\]
Conversely, each polarized product in \eqref{eq:regular-product}, with the
displayed product metric, has a regular quantization.
\end{corollary}

\begin{proof}
Regularity makes every coefficient of the Tian--Yau--Zelditch expansion of
the distortion function constant \cite[Lemma~2.3]{ArezzoLoiZuddas2012}; in
particular the first coefficient, a constant multiple of the scalar
curvature, is constant \cite{ArezzoLoiZuddas2013}.  Lemma~\ref{lem:cscK-first-chern-ray} therefore
gives $\Ric(\omega)=\lambda\omega$.  Choose a sufficiently large $r$ for
which $L^r$ is very ample and $r\omega$ is balanced.  Its coherent-states
map is a full K\"ahler immersion and
$\Ric(r\omega)=(\lambda/r)(r\omega)$.  Apply
Theorem~\ref{thm:intro-global}.  Since the Picard group of a product of
projective spaces is torsion-free, the conclusion for $L^r$ divides by $r$
to give \eqref{eq:regular-product}, and the metric identity divides by $r$.
Conversely, transitivity makes the distortion function of every positive
tensor power of the displayed homogeneous polarization constant.
\end{proof}

\begin{corollary}[Homothetic balanced product rigidity]
\label{cor:homothetic-balanced}
Let $(X,L)$ be a smooth compact polarized toric manifold with
$c_1(X)=\lambda c_1(L)$ for some $\lambda>0$.  Suppose that a balanced
K\"ahler metric $g_B$, polarized by a positive tensor power of $L$, has
infinitely many balanced positive-integer homotheties.  Then $X$ is a product
of projective spaces, $L$ is a matched product polarization as in
\eqref{eq:regular-product}, and $g_B$ is homogeneous.  Conversely, the
standard matched product metrics have infinitely many balanced
positive-integer homotheties.
\end{corollary}

\begin{proof}
By \cite[Lemma~2.3]{ArezzoLoiZuddas2012}, infinitely many balanced
homotheties force the scalar curvature to be constant.  If the K\"ahler form
$\omega_B$ of $g_B$ represents $2\pi c_1(L^{r_0})$, then
$c_1(X)=(\lambda/r_0)[\omega_B]/2\pi$.
Lemma~\ref{lem:cscK-first-chern-ray} makes $g_B$ K\"ahler--Einstein.  Choose
an arbitrarily large balanced homothety whose polarizing line bundle is very
ample and apply Theorem~\ref{thm:intro-global}.  Division in the torsion-free
Picard group gives the assertion for $L$.  The converse follows from
homogeneity, or directly from the complete product Veronese--Segre systems.
\end{proof}

For the anticanonical polarization $L=K_X^{-1}$,
Corollary~\ref{cor:homothetic-balanced} extends to all dimensions the
product-classification clause of
\cite[Theorem~1.2]{ArezzoLoiZuddas2012}.  Our results do not address the
existence assertion of that theorem.

\begin{corollary}[Finite-projective toric K\"ahler--Ricci solitons]
\label{cor:toric-krs}
Let $X$ be a smooth compact toric manifold and let $(\omega,V)$ be a
K\"ahler--Ricci soliton on $X$.  If $(X,\omega)$ admits a K\"ahler immersion
into a finite-dimensional complex projective space with its positive
Fubini--Study metric, then the soliton metric is K\"ahler--Einstein and
homogeneous.  After replacing the target by the projective span of the image,
all the conclusions of Theorem~\ref{thm:intro-global} hold.
\end{corollary}

\begin{proof}
Loi and Mossa prove that a K\"ahler--Ricci soliton K\"ahler immersed into a
finite-dimensional definite or indefinite complex space form is
K\"ahler--Einstein \cite[Theorem~1.1]{LoiMossa2021}.  In the present compact
positive-projective setting, Hulin's theorem makes the Einstein constant
positive \cite{Hulin2000}.  Restrict the target to the projective span and
apply Theorem~\ref{thm:intro-global}.
\end{proof}

These corollaries settle restricted forms of several questions in the
literature.  Theorem~\ref{thm:intro-global} proves the compact smooth toric
case of \cite[Conjecture~4.3.3]{LoiZedda2018}.  Corollary~\ref{cor:cscK-positive-ray} proves the compact smooth toric,
positive-first-Chern-ray case of the conjectural homogeneity of compact
projectively induced cscK metrics recalled in \cite{LoiZuddas2020}.
Corollary~\ref{cor:regular-quantization} proves the corresponding
positive-first-Chern-ray case of the toric product prediction in
\cite[Remark~3.3]{LoiZuddas2024}.  The unrestricted cscK, regular
quantization, complete noncompact, and nontoric questions remain open.

\section{Finite-support exponential families}
\label{sec:finite-nef}

This section proves Theorem~\ref{thm:intro-nef}.  Because the carrier $A$ is
an arbitrary finite subset of $\R^d$, there is no lattice or torus to start
from, and the first task is to build them.  Section~\ref{subsec:nef-germ}
does this by completing the projective exponential map.
Section~\ref{subsec:nef-classification} then applies
Corollary~\ref{cor:intro-fixed-torus} and restates the result in statistical
language, and Section~\ref{subsec:nef-literature} relates it to earlier work.

We first recall the statistical setting.  As in the introduction,
\[
 \psi(\theta)=\log\sum_{a\in A}c_ae^{\langle a,\theta\rangle}
\]
is the cumulant function of a minimal finite-support natural exponential
family, where $c_a>0$ and $A$ affinely spans $\R^d$.  Its Fisher information
in natural coordinates is $\nabla^2\psi$.  A standard Diaconis--Ylvisaker
conjugate prior has the form
\[
 \exp\{t(\langle\theta,m_0\rangle-\psi(\theta))\}\,d\theta,
 \qquad t>0
\]
\cite{DiaconisYlvisaker1979}.  Thus the Jeffreys prior belongs to this family
exactly when \eqref{eq:intro-fisher-determinant} holds, with
$(t,m_0)=(\lambda/2,b/\lambda)$; by Lemma~\ref{lem:nef-boundary}, $\lambda>0$
and this prior is proper.
Here $m_0$ is a mean hyperparameter; it is unrelated to the support-translation
point also denoted $m_0$ in Section~\ref{sec:preliminaries}.

\subsection{Boundary behavior and the projective germ}
\label{subsec:nef-germ}

Two facts come before the classification.  The first is elementary: the
identity forces $\lambda>0$ and makes the Jeffreys prior proper.  The second
manufactures the lattice and the torus that the hypothesis does not provide.

\begin{lemma}[Boundary and properness]
\label{lem:nef-boundary}
Under \eqref{eq:intro-fisher-determinant},
\[
 \lambda>0,\qquad \frac b\lambda\in\operatorname{int}\conv(A),
\]
and the corresponding Diaconis--Ylvisaker density is proper.
\end{lemma}

\begin{proof}
Define
\[
 p_\theta(a)=\frac{c_ae^{\langle a,\theta\rangle}}{Z(\theta)}.
\]
Then $\nabla\psi(\theta)=\mathbb E_\theta a$ and
$\nabla^2\psi(\theta)=\operatorname{Cov}_\theta(a)>0$.  Fix $u\ne0$,
put $P=\conv(A)$, and let $F_u$ be the carrier points on the exposed face
where $\langle a,u\rangle=h_P(u):=\max_{a'\in A}\langle a',u\rangle$.  As $t\to+\infty$,
\[
 p_{tu}(a)\longrightarrow
 \frac{c_a\mathbf1_{F_u}(a)}{\sum_{v\in F_u}c_v},\qquad
 \psi(tu)=t h_P(u)+\log\sum_{a\in F_u}c_a+o(1).
\]
The limiting law is supported on a proper affine hyperplane, so the limiting
covariance is singular.  Equation~\eqref{eq:intro-fisher-determinant} therefore
forces
\[
 \langle b,u\rangle<\lambda h_P(u)\qquad(u\ne0).
\]
Applying this to $-u$ gives
\[
 \lambda\min_{a\in A}\langle a,u\rangle
 <\langle b,u\rangle
 <\lambda\max_{a\in A}\langle a,u\rangle.
\]
Minimality makes the extrema distinct.  Hence $\lambda>0$, and the strict
support-function criterion gives $b/\lambda\in\operatorname{int}P$.

For completeness, compactness of the unit sphere gives $\delta>0$ such that
$h_P(u)-\langle b/\lambda,u\rangle\ge\delta$ when $\|u\|=1$.
The log-density of the conjugate prior therefore tends to $-\infty$ at least
linearly in $\|\theta\|$, proving integrability.
\end{proof}

\begin{lemma}[Intrinsic phase torus]
\label{lem:intrinsic-phase-torus}
Assume $d\ge2$ and \eqref{eq:intro-fisher-determinant}.  The holomorphic map
\begin{equation}
 \Phi(z)=\bigl[\sqrt{c_a}e^{\langle a,z\rangle}\bigr]_{a\in A}
 \colon\C^d\longrightarrow\mathbb P^{|A|-1}
 \label{eq:nef-projective-map}
\end{equation}
has a full compact smooth projective K\"ahler--Einstein completion $X$.
Moreover, there is an effective Hamiltonian $d$-torus $T$ on $X$ for which
the completed immersion is toric, and, for any $a_0\in A$,
\[
 A-a_0\subset M:=\operatorname{Hom}(T,S^1),\qquad
 \Z(A-a_0)=M.
\]
\end{lemma}

\begin{proof}
The pulled-back Fubini--Study potential in \eqref{eq:nef-projective-map} is
$\psi(z+\bar z)$.  Hence its Hermitian matrix is
$\nabla^2\psi(z+\bar z)>0$, and
\[
 \Ric(\Phi^*\omega_{\FS})=\lambda\Phi^*\omega_{\FS};
\]
the constant and affine terms in
\eqref{eq:intro-fisher-determinant} are pluriharmonic.  The immersion is full
because the exponentials $e^{\langle a,z\rangle}$ are linearly independent,
as one sees by restricting a putative relation to a generic complex line.

After shrinking to an embedded germ, Hulin's completion theorem
\cite[main theorem and Proposition~4.5]{Hulin1996} gives a complete
real-analytic projective Einstein continuation in the same projective space.
Up to coverings, its completed geometric image is unique.  Hulin's no-double-point conclusion lets us identify it with a smooth
submanifold $X\subset\mathbb P^{|A|-1}$; it is compact because
$\lambda>0$, and it remains full.

For $t\in\R^d$, set
\[
 \rho(t)=\left[\operatorname{diag}
 \bigl(e^{i\langle a,t\rangle}\bigr)_{a\in A}\right]
 \in\operatorname{PU}(|A|).
\]
On $\C^d$, $\rho(t)\Phi(z)=\Phi(z+it)$.  Choose connected neighborhoods
$U\Subset V$ of $0$ on which $\Phi$ parametrizes the embedded germ in $X$.
For all sufficiently small $t$, one has $U+it\subset V$, and
\[
 \rho(t)\Phi(U)=\Phi(U+it)
\]
is a nonempty open submanifold of both $\rho(t)X$ and $X$.  The two connected
complex submanifolds therefore agree by analytic continuation.  The set of
$t$ preserving $X$ is a subgroup containing a neighborhood of zero, hence is
all of $\R^d$.

Let $T=\overline{\rho(\R^d)}$.  It is a compact connected torus preserving
$X$ and acts effectively.  Indeed, a projective unitary transformation that
fixes $X$ pointwise has $X$ in the union of its projectivized eigenspaces;
irreducibility and fullness force it to be scalar.  If $r=\dim T$, then
$r\ge d$: an equality $d\rho_0(v)=0$ with $v\ne0$ would make
$\langle a-a_0,v\rangle=0$ for every $a\in A$, contrary to minimality.

Positive Ricci curvature and the Bochner identity give
$H^1(X;\R)=0$, so the $T$-action is Hamiltonian.  At a principal point its
orbit has dimension $r$, by effectivity, and Hamiltonian torus orbits are
isotropic.  Thus $r\le d$, and consequently $r=d$.  The differential of
$\rho\colon\R^d\to T$ is now an isomorphism; its image is open and hence all
of $T$.  Its kernel $\Lambda$ is a full lattice and
$T\simeq\R^d/\Lambda$.  A principal $T$-orbit is now $d$-dimensional and
isotropic, hence Lagrangian.  The projective stabilizer of $X$ is algebraic;
because it contains $T$, it contains the algebraic complexification
$T_\C\simeq(\C^*)^d$.  The complexified orbit through a principal point has
complex dimension $d$ and is therefore open.  Thus $X$ is a smooth compact
toric manifold.

If $t\in\Lambda$, the projective transformation $\rho(t)$ is scalar.  Hence
$\langle a-a_0,t\rangle\in2\pi\Z$, which proves $A-a_0\subset M$.  Put
$L=\Z(A-a_0)$.  If $L$ had finite index greater than one in $M$, the
nontrivial annihilator $\operatorname{Hom}(M/L,S^1)\subset T$ would act by a
common scalar on every homogeneous coordinate, contradicting effectivity.
Thus $L=M$.

After fixing $a_0$, the action has the linear diagonal lift
\[
 [t]\longmapsto\operatorname{diag}
 \bigl(e^{i\langle a-a_0,t\rangle}\bigr)_{a\in A},
 \qquad [t]\in\R^d/\Lambda.
\]
All coordinates of $\Phi(0)$ are nonzero, and the weights $A-a_0$ generate
$M$, so its $T_\C$-stabilizer is trivial.  Hence $\Phi(0)$ lies on the open
complex-torus orbit, and the original carrier is the toric polynomial
$\sum_{a\in A}c_a\chi^{a-a_0}$ in the full intrinsic character lattice.
\end{proof}

\subsection{Classification and converse}
\label{subsec:nef-classification}

With the intrinsic torus in hand, the case $d\ge2$ follows from the compact
theory.  The case $d=1$ reduces to an elementary differential equation, and
the converse is a short computation.

\begin{proof}[Proof of Theorem~\ref{thm:intro-nef}]
Lemma~\ref{lem:nef-boundary} proves positivity and properness.  Suppose first
that $d\ge2$.  By Lemma~\ref{lem:intrinsic-phase-torus}, the original
coordinates give a full toric immersion of the compact smooth toric manifold
$X$, its metric is invariant under the intrinsic torus, and its carrier
differences generate the full character lattice.  Corollary
\ref{cor:intro-fixed-torus} therefore gives
\[
 \sum_{a\in A}c_a\chi^{a-a_0}
 =Kx^w\prod_j
 \left(1+\sum_{q=1}^{n_j}\gamma_{jq}x_{jq}\right)^{m_j},
 \qquad \frac{n_j+1}{m_j}=\lambda.
\]
The multinomial theorem identifies the complete product carrier and its
weights:
\[
 K\prod_j
 \binom{m_j}{r_{j0},r_{j1},\ldots,r_{jn_j}}
 \prod_{j,q}\gamma_{jq}^{r_{jq}},
 \qquad r_{j0}=m_j-\sum_qr_{jq}.
\]
The last product is an exponential tilt, and the character coordinates give
an invertible real affine change of sufficient statistic.  This proves
\eqref{eq:intro-nef-factorization} for $d\ge2$.

Suppose now that $d=1$.  Let $\alpha=\min A$, $\beta=\max A$,
$m(\theta)=\psi'(\theta)$, and
$V(m(\theta))=\psi''(\theta)$.  Logarithmic differentiation of
\eqref{eq:intro-fisher-determinant} gives
\[
 V'(m)=\frac{\psi'''(\theta)}{\psi''(\theta)}=b-\lambda m.
\]
Since $m'(\theta)=\psi''(\theta)>0$, the mean map is a diffeomorphism from
$\R$ onto $(\alpha,\beta)$.  Exponential concentration at the two extreme
carrier points extends $V$ continuously to
$V(\alpha)=V(\beta)=0$.  Therefore
\[
 V(m)=\frac\lambda2(m-\alpha)(\beta-m).
\]
Put $\ell=2/\lambda$, $h=(\beta-\alpha)/\ell$, and
$y=(m-\alpha)/(\beta-\alpha)$.  Then $y'=hy(1-y)$, whence
\begin{equation}
 Z(\theta)=K e^{\alpha\theta}(1+\gamma e^{h\theta})^\ell
 \label{eq:nef-one-dimensional-Z}
\end{equation}
for some $K,\gamma>0$.  With $x=\gamma e^{h\theta}$, the left side after
division by $Ke^{\alpha\theta}$ is a finite positive generalized power sum,
whereas the right side is $(1+x)^\ell$.  Analyticity at zero forces every
exponent to be a nonnegative integer: after subtracting the preceding integer
terms, a first noninteger exponent would give a noninteger order of
vanishing.  The left side is therefore a polynomial, so the binomial series
forces $\ell\in\Z_{>0}$.  Comparing coefficients in
\eqref{eq:nef-one-dimensional-Z} gives
\[
 A=\{\alpha,\alpha+h,\ldots,\alpha+\ell h\},\qquad
 c_{\alpha+rh}=K\binom{\ell}{r}\gamma^r.
\]
This is the $\ell$-trial binomial carrier up to an affine change and a tilt,
and $\lambda=2/\ell$.

Conversely, for one canonical block put
\[
 D_j=1+\sum_{q=1}^{n_j}e^{\theta_{jq}},\qquad
 \psi_j=m_j\log D_j.
\]
The matrix determinant lemma gives
\[
 \det\nabla^2\psi_j
 =m_j^{n_j}\exp\!\left(
   \sum_{q=1}^{n_j}\theta_{jq}
   -\frac{n_j+1}{m_j}\psi_j\right).
\]
For a product, the Hessian is block diagonal.  Under the matching rule,
\[
 \det\nabla^2\psi
 =\left(\prod_jm_j^{n_j}\right)
   \exp\!\left(\sum_{j,q}\theta_{jq}-\lambda\psi\right).
\]
For the general factorization \eqref{eq:intro-nef-factorization}, let $B$ be
the matrix whose rows are the $u_{jq}^{\mathsf T}$.  The exact parameters are
\[
 b=\lambda v+\sum_{j,q}u_{jq},\qquad
 C=(\det B)^2\left(\prod_jm_j^{n_j}\right)
   \left(\prod_{j,q}\gamma_{jq}\right)K^\lambda.
\]
Common rescaling, exponential tilting, and an invertible affine change of
sufficient statistic preserve the form of the identity.  Explicitly, under
$a\mapsto La+v$,
\[
 \widetilde\psi(\vartheta)=\langle v,\vartheta\rangle
   +\psi(L^{\mathsf T}\vartheta),\qquad
 \det\nabla^2\widetilde\psi=(\det L)^2
   \det\nabla^2\psi(L^{\mathsf T}\vartheta),
\]
so $\lambda$ is unchanged.  This proves the converse.  Finally,
Lemma~\ref{lem:nef-boundary} shows that the Jeffreys prior is precisely the
proper Diaconis--Ylvisaker prior with precision $\lambda/2$ and mean
hyperparameter $b/\lambda$.
\end{proof}

\begin{corollary}[Exponential-family and type formulation]
\label{cor:nef-statistics-vocabulary}
Let $\mu_F$ be a positive measure on $\R^d$ with finite support, not
concentrated on an affine hyperplane, and let $F=F(\mu_F)$ be the natural
exponential family that it generates.  In the notation of
Theorem~\ref{thm:intro-nef}, $\mu_F=\sum_{a\in A}c_a\delta_a$, its Laplace
transform is $L_{\mu_F}=Z$, and its cumulant function is $\psi=\log Z$.  Let
$M_F=\operatorname{int}\conv(A)$ be the mean domain, let
$\theta(m)=(\nabla\psi)^{-1}(m)$ be the natural parameter as a function of the
mean $m\in M_F$, and let $V_F(m)=\nabla^2\psi(\theta(m))$ be the variance
function.  Let $P^*$ be the family of images under $\nabla\psi$ of the
normalized Diaconis--Ylvisaker priors
$\exp\{t(\langle\theta,m_0\rangle-\psi(\theta))\}\,d\theta$ with $t>0$ and
$m_0\in M_F$, and let $\widetilde P$ be the family of normalized priors on
$M_F$ with densities
$\exp\{t(\langle\theta(m),m_0\rangle-\psi(\theta(m)))\}\,dm$, whenever these
are integrable \cite{ConsonniVeronese1992,Casalis1996}.  Then the following
are equivalent.
\renewcommand{\labelenumi}{(\arabic{enumi})}
\begin{enumerate}
\item $\widetilde P=P^*$.
\item There are $b\in\R^d$ and $\lambda,c\in\R$ such that
 \[
  \det V_F(m)=\exp\{\langle\theta(m),b\rangle-\lambda\,\psi(\theta(m))+c\}
  \qquad(m\in M_F);
 \]
 equivalently, \eqref{eq:intro-fisher-determinant} holds with $C=e^c$.
\item There are $b\in\R^d$ and $\lambda\in\R$ such that the skewness vector
 satisfies
 \[
  \sum_{i=1}^dV_F'(m)(e_i)\,e_i=b-\lambda m\qquad(m\in M_F),
 \]
 where $(e_i)$ is the standard basis.
\item The Jeffreys prior $\sqrt{\det\nabla^2\psi(\theta)}\,d\theta$ is
 proportional to a Diaconis--Ylvisaker prior
 \[
  \exp\{t(\langle\theta,m_0\rangle-\psi(\theta))\}\,d\theta,
  \qquad t>0,\quad m_0\in M_F.
 \]
\item There are $C>0$, $b\in\R^d$, and $\lambda\in\R$ such that the measure
 $\nu_0$ of Kokonendji--Seshadri satisfies
 \[
  \nu_0=C\,\delta_b*\mu_F^{*(d+1-\lambda)},
 \]
 where $\mu_F^{*s}$ denotes the positive measure whose Laplace transform is
 $L_{\mu_F}^s$.
\item There are positive integers $n_j,m_j$ with $\sum_jn_j=d$ and
 $(n_j+1)/m_j=\lambda$ for every $j$ such that, after an invertible affine
 change of statistic, $F$ is the natural exponential family of the
 independent product $\bigotimes_j\operatorname{Mult}(n_j+1,m_j)$, where
 $\operatorname{Mult}(n+1,m)$ denotes a multinomial law with $n+1$
 categories and $m$ trials.
\end{enumerate}
The constants $b$ and $\lambda$ are the same in (2), (3), and (5), $\lambda$
is the common ratio in (6), and the constant $C$ in (5) is $e^c$.  In (4) one
then has $t=\lambda/2$ and $m_0=b/\lambda$.
\end{corollary}

\begin{proof}
Casalis writes $(B,b)$ for our $(b,-\lambda)$ and states the equivalence of
(1), (2), and (3) for every natural exponential family on $\R^d$
\cite[Introduction, item~4, pp.~1830--1831]{Casalis1996}.  The vector in (3)
is the skewness vector $\nabla\log\det\nabla^2\psi(\theta)$ evaluated at
$\theta=\theta(m)$ \cite[Corollary~2.4]{Hassairi1999}, so (3) is the gradient
form of (2).  The equivalence of (2) and (4) is the definition of the
Diaconis--Ylvisaker family together with Lemma~\ref{lem:nef-boundary}.  For
the equivalence of (2) and (5), recall that Kokonendji and Seshadri define
$\nu_0$ as the image of
\[
 \frac1{(d+1)!}
 \det\begin{pmatrix}1&1&\cdots&1\\X_0&X_1&\cdots&X_d\end{pmatrix}^2
 \mu_F(dX_0)\cdots\mu_F(dX_d)
\]
under $(X_0,\ldots,X_d)\mapsto X_0+\cdots+X_d$, and prove
\[
 L_{\nu_0}(\theta)=L_{\mu_F}(\theta)^{d+1}\det\nabla^2\psi(\theta)
\]
\cite[Theorem~2.2]{KokonendjiSeshadri1996}.  Taking Laplace transforms turns
(2) into (5); the positive measure $\mu_F^{*(d+1-\lambda)}$ exists because its
Laplace transform is then that of $C^{-1}\delta_{-b}*\nu_0$.  Conversely, (5)
gives (2).  Finally, the equivalence of (2) and (6) is
Theorem~\ref{thm:intro-nef}.  For the matched product,
$(d+1)m_j-(n_j+1)\in\mathbb Z_{\ge0}$ in every block, so
$\mu_F^{*(d+1-\lambda)}$ is, block by block, an integral convolution power of
a one-trial categorical law, although $d+1-\lambda$ itself need not be an
integer.  The only degenerate case is the one-trial categorical family, where
$\lambda=d+1$ and $\nu_0$ is a point mass.
\end{proof}

\begin{remark}[Types, products, and support conventions]
\label{rem:nef-types}
Two natural exponential families are of the same type when one is the image
of a convolution power of the other under an invertible affine map.  Affine
changes of statistic are already allowed in
Corollary~\ref{cor:nef-statistics-vocabulary}(6); passing to types also
allows convolution powers, which replace every $m_j$ by $pm_j$ and $\lambda$
by $\lambda/p$.  Hence there is exactly one type for each partition
$d=\sum_jn_j$.  With $G=\gcd_j(n_j+1)$, it contains the representative
$m_j=(n_j+1)/G$, for which $\lambda=G$, and every member of the type is an
affine image of a positive integral convolution power of that representative.
Thus the numbers of types in dimensions $1,2,3,4$ are $1,2,3,5$.  Among these
types only the one-block multinomial is simple quadratic; the multi-block
products are not among the classical simple-quadratic and Wishart examples.

The scalar linear part in Corollary~\ref{cor:nef-statistics-vocabulary}(3)
is essential.  For Bernoulli$\times$binomial$(2)$, with statistic
$(x_1,x_2)\in\{0,1\}\times\{0,1,2\}$, the skewness vector is
$(1-2m_1,\,1-m_2)$, which is affine in $m$, but (2) fails because the two
blocks have $\lambda=2$ and $\lambda=1$.  More generally, an independent
product of families satisfying (2) satisfies (2) exactly when all factors
have the same $\lambda$; the vectors $b$ of the factors are then
concatenated, because $\det V_F$ is multiplicative and the cumulants add.

Item (5) is the translation-and-power form of the relation ``$F(\nu_0)$ and
$F(\mu_F)$ are of the same type'' proved by Kokonendji and Seshadri for the
simple-quadratic class \cite[Theorem~3.1]{KokonendjiSeshadri1996}; their
proof produces exactly a translation and a convolution power.  In general,
the same-type relation only requires $\nu_0$ to be, up to a positive constant
and an exponential tilt, the image of a convolution power of $\mu_F$ under
some invertible affine map.  We do not classify finite-support families under
that weaker relation.

Finally, finite support means a finitely supported generating measure,
whereas bounded support may include continuous or mixed components.  The
bounded-support result of Ghribi--Masmoudi \cite{GhribiMasmoudi2010} is
therefore adjacent rather than a special case of Theorem~\ref{thm:intro-nef}.
In dimension one Morris's classification \cite{Morris1982} reduces the bounded case to the binomial family;
higher-dimensional continuous or mixed bounded supports remain outside the
present proof.
\end{remark}

\subsection{Relation to the statistical literature}
\label{subsec:nef-literature}

We close this section by placing Theorem~\ref{thm:intro-nef} and
Corollary~\ref{cor:nef-statistics-vocabulary} among earlier results on the
same identity.  Equation~\eqref{eq:intro-fisher-determinant} has a
substantial statistics history.  Casalis formulates the equivalent
determinant condition while
comparing the standard conjugate families on the natural and mean parameter
spaces and notes that the full class satisfying it was not known
\cite[Introduction, item~4, pp.~1830--1831]{Casalis1996}.  She also points
out that Guti\'errez-Pe\~na and Smith independently obtained a similar
statement.  They
study when transformations of the canonical or mean parameter preserve the
standard conjugate form, give multivariate extensions, and relate the Jeffreys
prior to that family \cite{GutierrezPenaSmith1995,GutierrezPenaSmith1997}; see also
the published correction to the 1995 article
\cite{GutierrezPenaSmith1996Correction}.

Several results are closer to ours.  Kokonendji and Seshadri give the
determinant-weighted convolution identity behind \eqref{eq:cauchy-binet},
prove that $F(\nu_0)$ and $F(\mu_F)$ are of the same type for every simple
quadratic family \cite[Theorem~3.1]{KokonendjiSeshadri1996}, and explicitly
compute the multinomial determinant law
\cite[Theorem~2.2 and Section~3.1]{KokonendjiSeshadri1996}.
Consonni--Veronese introduced the comparison of the two conjugate families
in dimension one \cite{ConsonniVeronese1992}.  Hassairi identifies the vector
in Casalis's third criterion with the skewness vector and, for a class of
generating measures that includes the infinitely divisible ones, reformulates
$\widetilde P=P^*$ through his generalized variance transform
\cite[Corollaries~2.4 and~2.5]{Hassairi1999}.
Druilhet--Pommeret study invariant Jeffreys-conjugate priors
\cite{DruilhetPommeret2012}.  For Hessian geometry, see Shima
\cite{Shima2007}; Furuhata and Kurose classify Hessian manifolds of
nonpositive constant Hessian sectional curvature \cite{FuruhataKurose2013}.
Eaves and Chang prove, in dimension one, that the Jeffreys prior is conjugate
exactly for quadratic variance functions, and treat the multinomial family
separately \cite[Proposition~2.1]{EavesChang1992}.
Ghribi and Masmoudi prove a bounded-support converse for the normalized
one-trial categorical generalized-variance identity
\cite[Theorem~3.1]{GhribiMasmoudi2010}.  The one-dimensional binomial case
also belongs to the classical quadratic-variance classification
\cite{Morris1982}.

\section{The Manno--Salis germ and univalent classification}
\label{sec:local}

This section proves Corollary~\ref{cor:intro-manno-salis}.  The idea is to
complete the local Einstein germ to a compact submanifold, extend the torus
action to the completion, and apply Theorem~\ref{thm:intro-global} there.  We
then explain why univalence cannot be dropped and give a direct algebraic
check of the normalized Manno--Salis polynomial.

\subsection{Completion and extension of the torus action}

The proof has three steps: completion of the germ, extension of the torus
action, and restriction back to the original domain.  The case $d=1$ is
elementary and is treated first.

\begin{proof}[Proof of Corollary~\ref{cor:intro-manno-salis}]
When $d=1$, the germ assertion is the rank-one case of
\cite[Theorem~1.1]{MannoSalis2026}; equivalently, it follows from the
powered-binomial classification \cite[Proposition~4.1]{DiScalaSombra2025}
after local torification.  In the univalent case, the local image lies in
the corresponding rational normal curve.  The homogeneous equations of
that curve vanish on a nonempty open subset of the connected image and
hence everywhere; the immersion therefore identifies all of $Y$ with an
open subset of the curve.  This proves both assertions when $d=1$.
Assume henceforth that $d\ge2$.

Shrink to a connected invariant neighborhood $V$ of the fixed point on which
$\varphi$ is an embedding.  Hulin's completion theorem extends this
projective submanifold germ to a complete real-analytic K\"ahler--Einstein
submanifold $(\overline Y,\overline\omega)$ of the same projective space
\cite{Hulin1996}.  This is the completion step used in Lemma~2.13 of the
preprint \cite{MannoSalis2024}, numbered Lemma~2.14 in the published version
\cite{MannoSalis2026}.  The completed immersion remains full, since a
hyperplane containing its image would also contain the original open germ.
Since $\lambda>0$, Bonnet--Myers makes $\overline Y$
compact, and Kobayashi's theorem for compact K\"ahler manifolds with positive
Ricci tensor makes it simply connected \cite{Kobayashi1961}.

Next we check that the given torus action survives this completion.  Choose a
basis of the integral lattice of $T^d$ and let $\xi_1,\ldots,\xi_d$ be the
corresponding Killing fields near the fixed point.  Nomizu's extension
theorem uniquely extends each $\xi_i$ to a global Killing field
$\overline\xi_i$ on the simply connected real-analytic manifold
$\overline Y$ \cite{Nomizu1960}.  Their brackets vanish globally because
they vanish on a nonempty open set.  The extensions are real holomorphic:
$\mathcal L_{\overline\xi_i}J$ is analytic and vanishes on that same open set.
Compactness makes all their flows complete.

Normalize the integral lattice so that its elements have time-one flow.  For
each $\ell\in\Z^d$, the corresponding global time-one isometry is the identity
on $V$ and therefore on all of $\overline Y$: an isometry
of a connected Riemannian manifold is determined by its value and derivative
at one point.  The commuting $\R^d$-action consequently factors through
$\R^d/\Z^d=T^d$.  It is effective, since a global kernel element would lie in
the kernel of the original action, and it fixes the original fixed point.
Finally $H^1(\overline Y;\R)=0$.  Each closed one-form
$\iota_{\overline\xi_i}\overline\omega$ is therefore exact, so the action is
Hamiltonian.  Hence $(\overline Y,\overline\omega)$ is a smooth compact toric
K\"ahler manifold.  This expands the torus-extension step stated in
Lemma~2.14 of the preprint \cite{MannoSalis2024}, numbered Lemma~2.15 in the
published version \cite{MannoSalis2026}.

Theorem~\ref{thm:intro-global} applies to $\overline Y$.  Restricting its
holomorphic isometry to $V$ proves the germ assertion.  If $\varphi$ is
univalent in the sense stated in the corollary, the connected embedded image
$\varphi(Y)$ continues from $\varphi(V)$ inside the Hulin completion.  The
two submanifolds have the same dimension, so this inclusion is open; hence
the product identification restricts to all of $Y$.  In the Riemannian
normalization $\Ric(g)=2g$ of \cite[Conjecture~3]{MannoSalis2024}, our form
constant is $\lambda=1$, and the matching equation reads $m_j=n_j+1$.
\end{proof}

\begin{remark}[The covering qualification is necessary]
\label{rem:covering-obstruction}
The literal open-subset conclusion for an arbitrary abstract immersed domain
fails without a univalence condition.  Here is a concrete construction.  Let
$\nu\colon\mathbb P^2\to\Delta$ be the standard moment map, choose a small
closed disk $D\subset\operatorname{int}\Delta$, and put
$B=\Delta\setminus D$ and $U=\nu^{-1}(B)$.  Thus $B$ is relatively open in
$\Delta$, and $U$ is a connected
$T^2$-invariant open subset containing the torus fixed points.  The moment
map has the continuous section
\[
 (r_1,r_2)\longmapsto
 [\sqrt{1-r_1-r_2}:\sqrt{r_1}:\sqrt{r_2}],
\]
so $\pi_1(B)\simeq\Z$ injects into $\pi_1(U)$.  In particular the universal
cover $\pi\colon\widetilde U\to U$ is nontrivial.

Let $\mathcal V_3\colon\mathbb P^2\to\mathbb P^9$ be the complete cubic
Veronese immersion.  Pull back $3\omega_{\FS}$ and compose $\pi$ with
$U\hookrightarrow\mathbb P^2$ and $\mathcal V_3$.  This is a full K\"ahler
immersion of the K\"ahler--Einstein manifold $\widetilde U$ into
$\mathbb P^9$, now in the normalized regime
$\Ric(3\omega_{\FS})=3\omega_{\FS}=1\cdot(3\omega_{\FS})$.  The universal
cover $\R^2$ of $T^2$ acts on $\widetilde U$ by the uniquely lifted flows that
fix a chosen lift of a torus fixed point.  Every lattice element acts as a
deck transformation fixing that point and is therefore the identity, so the
action descends to $T^2$.  It remains effective and is Hamiltonian, with the
moment map $3\nu\circ\pi$ (up to an additive constant).  In complex
dimension two, the normalized product
models are $(\mathbb P^2,3\omega_{\FS})$ and
$(\mathbb P^1\times\mathbb P^1,
2\omega_{\FS}\oplus2\omega_{\FS})$.  If $\widetilde U$ were holomorphically
isometric to an open subset of the first, composing that identification with
$\mathcal V_3$ and applying Calabi rigidity in $\mathbb P^9$ would make the
covering immersion a unitary transform of an injective map, a contradiction.
For the second model, its complete $(2,2)$ Veronese--Segre immersion is full
in $\mathbb P^8$, whereas the covering immersion is full in $\mathbb P^9$;
Calabi rigidity forbids two full projective immersions of the same connected
K\"ahler manifold in different target dimensions.  Thus the literal global
open-subset form fails even under the normalized hypothesis of
Conjecture~3 in \cite{MannoSalis2024}.
This is the covering ambiguity explicitly retained in Hulin's completion
theorem \cite[p.~288]{Hulin1996}.
\end{remark}

\subsection{A direct normalized-polynomial certificate}

There is also a purely algebraic way to see why the normalized polynomial in
the Manno--Salis formulation enters the scope of the bivariate theorem.  Put
$\mathbf1=(1,\ldots,1)$.

\begin{lemma}[Cauchy--Binet vertex certificate]
\label{lem:vertex-certificate}
Let
\[
 P(x)=\sum_{a\in\mathcal A}c_ax^a\in\C[x_1^{\pm1},\ldots,x_d^{\pm1}]
\]
have effective rank $d$, with $0\in\mathcal A$ and every $c_a\ne0$, and set
$Q=\NP(P)$.  Compute $\mu(P)$ in \eqref{eq:def-mu} using the support
translation $m_0=0$, so that the Cauchy--Binet identity
\eqref{eq:cauchy-binet} holds with exactly the displayed exponents.  Suppose
\begin{equation}
 \mu(P)=x^{\mathbf1}P^d
 \label{eq:normalized-mu}
\end{equation}
and $R=Q-\mathbf1$ is Delzant and reflexive.  Then $\mathcal A$ is
unimodular.
\end{lemma}

\begin{proof}
Fix a vertex $v$ of $Q$, put $r=v-\mathbf1$, and let
$u_1,\ldots,u_d$ be the primitive outgoing edge basis of $R$ at $r$.  If
$\eta_1,\ldots,\eta_d$ are the primitive normals of the incident facets,
oriented so that their inequalities are at most one, smooth reflexivity gives
\[
 \langle\eta_i,r\rangle=1,
 \qquad
 \langle\eta_i,u_j\rangle=-\delta_{ij}.
\]
It follows that
\begin{equation}
 \mathbf1-v=-r=\sum_{i=1}^du_i.
 \label{eq:reflexive-vertex-sum}
\end{equation}

The coefficient of $x^{\mathbf1+dv}$ on the right of
\eqref{eq:normalized-mu} is $c_v^d\ne0$.  Indeed, choose a linear functional
uniquely minimized on $Q$ at $v$; if a sum of $d$ support points equals
$dv$, equality in the resulting $d$ lower bounds forces every summand to be
$v$.
The Cauchy--Binet expansion \eqref{eq:cauchy-binet} therefore contains an
affinely independent $(d+1)$-subset $B=\{a_0,\ldots,a_d\}\subset\mathcal A$
with
\begin{equation}
 \sum_{j=0}^da_j=\mathbf1+dv.
 \label{eq:cb-vertex-sum}
\end{equation}
Every lattice point in the tangent cone has a unique expression
\[
 a_j-v=\sum_{i=1}^dk_{ji}u_i,
 \qquad k_{ji}\in\Z_{\ge0}.
\]
Equations \eqref{eq:reflexive-vertex-sum} and
\eqref{eq:cb-vertex-sum} give $\sum_jk_{ji}=1$ for each $i$.  Thus there are
exactly $d$ unit entries among the $d+1$ distinct row vectors
$(k_{j1},\ldots,k_{jd})$.  Affine independence forces these rows to be
$0,e_1,\ldots,e_d$.  Hence
\[
 \{v,v+u_1,\ldots,v+u_d\}\subset\mathcal A.
\]
This holds at every vertex, which is precisely unimodularity of $\mathcal A$.
\end{proof}

\begin{corollary}[Normalized Manno--Salis polynomial]
\label{cor:normalized-polynomial}
Let $P=\sum_{a\in\mathcal A}c_ax^a\in\R[x_1,\ldots,x_d]$ have finite
full-rank support $\mathcal A\subset\mathbb N^d$, with $c_a>0$ for every
$a\in\mathcal A$ and
$c_0=c_{e_i}=1$, and suppose
\begin{equation}
 u(t)=\log P(e^{t_1},\ldots,e^{t_d})-\sum_{i=1}^dt_i,
 \qquad \det D^2u=e^{-u}.
 \label{eq:ms-normalized-equation}
\end{equation}
Then there is a partition
$\{1,\ldots,d\}=B_1\sqcup\cdots\sqcup B_k$, with $n_j=|B_j|$, such that
\begin{equation}
 P(x)=\prod_{j=1}^k
 \left(1+\frac{1}{n_j+1}\sum_{i\in B_j}x_i\right)^{n_j+1}.
 \label{eq:normalized-polynomial-form}
\end{equation}
\end{corollary}

\begin{proof}
Equation~\eqref{eq:ms-normalized-equation} is exactly
$\mu(P)=x^{\mathbf1}P^d$.  Proposition~2.15 of the preprint
\cite{MannoSalis2024}, numbered Proposition~2.16 in the published version
\cite{MannoSalis2026}, says that $R=\NP(P)-\mathbf1$ is Delzant and
reflexive.
Lemma~\ref{lem:vertex-certificate} makes $\supp(P)$ unimodular, while
\eqref{eq:normalized-mu} makes $P$ GEC.
When $d=1$, the same conclusion is obtained before invoking the rank-two
theorem: reflexivity gives $\NP(P)=[0,2]$, and the rank-one powered-binomial
classification \cite[Proposition~4.1]{DiScalaSombra2025}, together with
$c_0=c_1=1$, gives $P(x)=(1+x/2)^2$.  Thus the remainder of the argument may
be read with $d\ge2$.

For $d\ge2$, Theorem~\ref{thm:intro-surface}, Yu--Masuda's criterion
\cite[Theorem~2.1]{YuMasuda2021}, and
Lemma~\ref{lem:lattice-untwisting} show that $\NP(P)$ is a product of
dilated simplices.  Since $\NP(P)$ lies in the positive orthant and contains
$0,e_1,\ldots,e_d$, its tangent cone at $0$ is the positive orthant.  Thus
its primitive edges there are the coordinate rays, and the product factors
align with a partition of the variables.  Reflexivity of the translate
forces the dilation in an $n_j$-dimensional block to be $n_j+1$.

The closure of the gradient image of $u$ is $R$; the existence equation also
gives the zero-barycenter condition recorded in Lemma~2.16 of the preprint
\cite{MannoSalis2024}.  The right-hand side $P_0$ of
\eqref{eq:normalized-polynomial-form} gives a solution $u_0$ of
\eqref{eq:ms-normalized-equation} with the same gradient image.  Uniqueness
for this real Monge--Amp\`ere equation
\cite[Proposition~2.17]{MannoSalis2024} (Proposition~2.18 in
\cite{MannoSalis2026}) gives $u(t)=u_0(t+c)$ for some
$c\in\R^d$.  Hence
\[
 P(x)=e^{-\sum_i c_i}P_0(e^{c_1}x_1,\ldots,e^{c_d}x_d).
\]
The constant normalization first gives $\sum_i c_i=0$, and the linear
normalizations then give $c_i=0$ for every $i$.  This proves
\eqref{eq:normalized-polynomial-form}.
\end{proof}

\begin{remark}[Arbitrary Einstein constants]
For $\lambda\ne1$ one cannot reduce to
Corollary~\ref{cor:normalized-polynomial} by rescaling
(Remark~\ref{rem:normalization-caveat}); the statements for arbitrary
Einstein constants therefore use completion and
Theorem~\ref{thm:intro-global}.
\end{remark}

\section{Further comparisons, limitations, and open problems}
\label{sec:scope}

This final section looks at the results from outside.  We compare
Theorem~\ref{thm:intro-surface} with linear precision in algebraic
statistics, compare the geometric results with the work of Manno and Salis,
and end with the limitations of the method and some open problems.

Two remarks on the algebraic side come first.  Before the present
classification, Di Scala and Sombra had excluded particular trapezoidal and
hexagonal faces and used those exclusions in several higher-dimensional
families \cite{DiScalaSombra2025}; the signed ledger replaces such case
analysis by one uniform argument.  Also, the exponent three in
$\mu(p)\mid p^3$ is a consequence of the classification, not of a general
derivative-order estimate.

\subsection{Likelihood maps and linear precision}

The polygons in Theorem~\ref{thm:intro-surface} also appear in the theory of
toric patches with linear precision.  The two conditions are nevertheless
different, and this subsection makes the comparison precise.  Write
\[
 \mathcal L_p(x,y)=
 \left(\frac{x\partial_xp}{p},\frac{y\partial_yp}{p}\right)
\]
for the logarithmic polar, or likelihood, map.  Rational linear precision of
the associated toric patch is equivalent to birationality of this map and
hence to maximum-likelihood degree one
\cite[Theorem~3.9 and Proposition~4.1]{GarciaPuenteSottile2010}; see also
\cite[Sections~5, 6, and 8]{ClarkeCox2020}.  Our unimodularity hypothesis
makes the exponent configuration primitive.  For coefficients outside the
positive real locus, ``ML degree'' below denotes the algebraic degree of
$\mathcal L_p$.

\begin{proposition}[Rank-two GEC implies ML degree one]
\label{prop:gec-ml-degree}
Every rank-two unimodular GEC polynomial in
Theorem~\ref{thm:intro-surface} has maximum-likelihood degree one.  The
converse is false, even for a positive polynomial with full lattice-point
support in a smooth polygon.
\end{proposition}

\begin{proof}
Multiplication by a Laurent unit translates $\mathcal L_p$, and an integral
monomial change conjugates it by invertible integral linear maps.  For
$p=(\alpha_0+\alpha_1x+\alpha_2y)^m$,
\[
 \mathcal L_p(x,y)=
 \left(\frac{m\alpha_1x}{\alpha_0+\alpha_1x+\alpha_2y},
       \frac{m\alpha_2y}{\alpha_0+\alpha_1x+\alpha_2y}\right),
\]
whose rational inverse is
\[
 x=\frac{\alpha_0u}{\alpha_1(m-u-v)},\qquad
 y=\frac{\alpha_0v}{\alpha_2(m-u-v)}.
\]
For $p=(\alpha_0+\alpha_1x)^r(\beta_0+\beta_1y)^s$, the two coordinates
separate and the inverse is
\[
 x=\frac{\alpha_0u}{\alpha_1(r-u)},\qquad
 y=\frac{\beta_0v}{\beta_1(s-v)}.
\]
Theorem~\ref{thm:intro-surface} proves the first assertion.

For the converse, take
\[
 p=(1+x)(1+x+y)=1+2x+x^2+y+xy.
\]
Its support is the complete lattice-point set of the smooth trapezoid
$\conv\{(0,0),(2,0),(1,1),(0,1)\}$, hence is unimodular.  Direct calculation
gives
\[
 \mu(p)=xy(1+x)^2(2+2x+y),
\]
and the last factor is coprime to $p$, so $p$ is not GEC.  On the other hand,
if $(u,v)=\mathcal L_p(x,y)$, then
\[
 x=\frac{u}{2-u-v},\qquad
 y=\frac{v(2-v)}{(1-v)(2-u-v)},
\]
which proves birationality.
\end{proof}

\begin{remark}[Nonprimitive parametrizations]
Without unimodularity the first assertion of
Proposition~\ref{prop:gec-ml-degree} fails for a trivial reason: $p=1+x^2+y$
satisfies $\mu(p)\mid p^N$, but its logarithmic polar map has degree two.  Its
support differences generate the index-two lattice $2\Z\oplus\Z$, and the
polar map factors through the isogeny $x\mapsto x^2$.  On the faithful
character lattice, with $w=x^2$, the same toric model is represented by
$1+w+y$, which has ML degree one.
\end{remark}

Strict linear precision implies rational linear precision.  Clarke and Cox
prove that the lattice polygons admitting suitable positive weights with
strict linear precision are
$m\Sigma_2$ and $r\Sigma_1\times s\Sigma_1$
\cite[Theorem~4.9]{ClarkeCox2020}.  Thus their list agrees with the
Newton-polygon list in Theorem~\ref{thm:intro-surface}; this is an agreement
of polygon shapes, not an identification of conditions on a fixed weighted
polynomial.  For a fixed primitive positive weighting, Proposition
\ref{prop:gec-ml-degree} shows that GEC implies rational linear precision,
equivalently ML degree one, while its trapezoidal example shows that the
converse fails.  That trapezoid is also the standard example of rational but
not strict linear precision.  The broader toric-polar classification contains
additional trapezoidal and conic families
\cite[Theorem~1 and Corollary~2]{BothmerRanestadSottile2010}.

In higher dimensions Clarke and Cox conjecture that the polytopes with
strict linear precision are the B\'ezier simploids
$\prod_jm_j\Sigma_{n_j}$ with arbitrary positive dilations
\cite[Conjecture~4.8]{ClarkeCox2020}.  Our compact K\"ahler--Einstein theorem
reaches only the smaller subclass satisfying the Einstein matching rule and
starts from a different hypothesis; it does not resolve that conjecture.

\subsection{Comparison with the work of Manno and Salis}

Manno and Salis have studied the same local problem in a normalized form, and
several of our statements refine or qualify theirs.  At the polynomial and
fixed-point-germ level, they already obtain the
positive surface classification
\cite[Theorem~1.6 and Proposition~2.6]{MannoSalis2022}; their global
open-subset statement is subject to the covering qualification in
Remark~\ref{rem:covering-obstruction}.  In the para-K\"ahler setting they
also classify the real two-variable polynomial solutions arising from their
normalized projective form, allowing either sign
\cite[Proposition~4.5]{MannoSalisPara2025}.  By contrast,
Theorem~\ref{thm:intro-surface} treats arbitrary complex coefficients under
the weaker divisibility condition GEC.

In higher dimensions, the published classification of Manno and Salis
reaches complex dimension six
\cite[Theorem~1.1]{MannoSalis2026}; see also the earlier input of
Arezzo--Loi--Zuddas \cite[Proposition~4.2]{ArezzoLoiZuddas2012}.
Theorem~\ref{thm:intro-global} removes the dimension bound for every smooth
compact toric manifold and every full projective immersion.  We state the
Einstein matching invariantly as $(n_j+1)/m_j=\lambda$: for unequal factor
dimensions, the coefficients displayed in \cite[Theorem~1.1]{MannoSalis2026}
and \cite[Conjecture~2]{MannoSalis2024} are reciprocal to it, giving the
weights $(3,2)$ instead of $(2,3)$ for $\mathbb P^1\times\mathbb P^2$.

The automorphism $F$ in Theorem~\ref{thm:intro-global} cannot be omitted in
general, because an arbitrary K\"ahler--Einstein metric need not be invariant
under the torus we started with.  When the metric is invariant,
Corollary~\ref{cor:intro-fixed-torus} says more: even a monomial system with
omitted lattice points is forced to acquire the full multinomial carrier.

On the local side, the completion and torus-extension steps are those of
Hulin and of Manno and Salis \cite[Lemmas~2.13 and~2.14]{MannoSalis2024}.
The periodicity argument in Section~\ref{sec:local} makes explicit that the
extended $\R^d$-action descends to the original torus.  The normalized
polynomial Corollary~\ref{cor:normalized-polynomial} supplies a second,
algebraic view of the same closure: its new vertex certificate upgrades
Delzant reflexivity to unimodularity of the actual support.  The literal
global assertion for an arbitrary abstract immersed domain is false without
a no-monodromy hypothesis, as Remark~\ref{rem:covering-obstruction} shows.

Finally, projective homogeneous toric varieties are products of projective
spaces \cite{ArzhantsevGaifullin2010}, so Theorem~\ref{thm:intro-global} is
the conclusion that the homogeneity conjecture predicts for compact toric
manifolds.  The proof does not assume homogeneity: it derives it from the
two-dimensional GEC obstruction.

\subsection{Limitations and open problems}

We end with what the method does not reach and the questions it suggests.
Delzant smoothness is used in the
surface self-intersection formulas, the vertex certificate, and lattice
untwisting.  Singular toric varieties and orbifolds would require
index-corrected ledgers.  More fundamentally, the argument needs a torus of
half the real dimension; it does not address the general Loi--Zedda
homogeneity conjecture for projectively induced K\"ahler--Einstein manifolds with a smaller symmetry
group or no torus action.

The following problems appear most direct.
\begin{enumerate}
 \item Extend the signed ledger and the Cauchy--Binet vertex certificate to
 simplicial orbifold polygons.  The
 determinants in \eqref{eq:cauchy-binet}, wall relations, and exposed lattice
 lengths should then carry local index factors.
 \item Classify higher-rank GEC polynomials directly.  The present geometric
 application needs only two-faces, but a factorization theorem in rank three
 or higher could distinguish GEC from other logarithmic-Hessian and toric
 polar conditions without assuming smooth complete compactifications.
 \item Find a replacement for the face-by-face toric reduction under smaller
 symmetry groups.  This is the obstruction between
 Corollary~\ref{cor:intro-manno-salis} and the unrestricted homogeneity
 conjecture for projectively induced K\"ahler--Einstein manifolds.
 \item For $d\ge2$, determine the infinite-support natural exponential
 families satisfying the Fisher-determinant identity
 \eqref{eq:intro-fisher-determinant}; for $d=1$ they are Morris's
 quadratic-variance families.
 Theorem~\ref{thm:intro-nef} settles the entire finite-support class, but its
 projective compactification argument has no direct analogue for a general
 Laplace transform.
\end{enumerate}

\paragraph{AI-use disclosure.}
The author used Anthropic Claude Code and OpenAI Codex as interactive
research and writing tools. They assisted with exploratory discussion,
testing and refinement of ideas, literature and source organization, code
development and verification, mathematical error checking, and editorial
revision.

\bibliographystyle{alpha}
\bibliography{references}

\end{document}